\documentclass[12pt]{article}
\usepackage{float,amsmath,amsthm,amssymb,
latexsym,graphicx,amsfonts,amstext,amscd,
amsbsy,bezier,enumerate,bigints,subfigure,multicol,graphicx}

\newcommand{\bb}[1]{\mathbf{#1}}
\theoremstyle{plain}
\newtheorem{Proposition}{Proposition}
\newtheorem{Theorem}{Theorem}
\newtheorem{Remark}{Remark}

\newtheorem{Lemma}{Lemma}
\newtheorem{Definition}{Definition}

\newcommand{\cqd}{\hfill \rule{2mm}{2mm}}

\title{{\bf\Large On stability properties of the Cubic-Quintic Schr\"odinger equation with $\delta$-point interaction}}
\author{{\bf\large Jaime Angulo Pava}\footnote{Email: angulo@ime.usp.br, Tel. (55)-11-3091-6265, Fax: (55)-11-3091-6183}\hspace{2mm}
{\bf\large}\vspace{1mm}\\
{\it\small Department of Mathematics, IME-USP}\\
 {\it\small Rua do Mat\~ao 1010, Cidade Universit\'aria,}\\
{\it\small  CEP 05508-090, S\~ao Paulo, SP, Brazil}
\vspace{3mm}\\
{\bf\large C\'esar A. Hern\'andez Melo}\footnote{Email:  cahmelo@uem.br, Tel. (55)-44-3011-5358, Fax: (55)-11-3011-3873}\hspace{2mm}
{\bf\large}\vspace{1mm}\\
{\it\small Department of Mathematics, DMA-UEM}\\
{\it\small Av. Colombo, 5790 Jd. Universit\'ario,}\\
{\it\small  CEP 87020-900, Maring\'a, PR, Brazil}\vspace{3mm}}

\date{May 13, 2018}
\begin{document}
\maketitle
\begin{abstract}
 We study analytically and numerically  the existence and orbital stability of the peak-standing-wave solutions for
 the  cubic-quintic nonlinear Schr\"odinger equation with a point interaction determined by the delta of Dirac. We study the cases  of  attractive-attractive and attractive-repulsive nonlinearities and we recover some results in the literature. Via a perturbation method and continuation argument we determine the Morse index of some specific self-adjoint operators that arise in the stability study. Orbital instability implications from a spectral instability result are established. In the case of an attractive-attractive case and an focusing interaction we give an approach based in the extension theory of symmetric operators for determining the Morse index.
\end{abstract}

\textbf{Key  words}. Nonlinear stability of peak-solutions, nonlinear Schr\"odinger equation with point interaction, analytic perturbation theory, self-adjoint extensions.

\section{Introduction}\label{int}

This work addresses the nonlinear stability and instability properties of peak-standing wave associated to the following cubic-quintic nonlinear Schr\"odinger equation with a point interaction determined by Dirac $\delta$ distribution centered  at the origem (henceforth NLSCQ),
\begin{equation}\label{deltasch1}
iu_{t}+u_{xx}+Z\delta(x)u+\lambda_1|u|^2u+\lambda_2 |u|^4u=0,\quad\text{for}\;\;t, x\in \mathbb{R},
\end{equation}
here $u=u(x,t)\in \mathbb C$, $\lambda_i\in \mathbb R$, and  $Z\in \mathbb{R}$ represents the so-called strength parameter.

Equation NLSCQ represents a very general  family of models featuring the competition between  repulsive ($\lambda_1\leqq 0$)/attractive ($\lambda_1\geqq 0$) cubic and repulsive ($\lambda_2\leqq 0$)/attractive ($\lambda_2\geqq 0$)/quintic terms, together with a point interaction or defect determined by the $\delta$ distribution at the origem. That models with $Z=0$, it have drawn considerable attention in the last years by its  physical relevance in optical media, liquid waveguides and others (\cite{BouCLT}, \cite{FilhoABL}, \cite{FilhoAR}, \cite{GiDriMa}). The case $\lambda_2=0$ and $Z\neq 0$ has been also studied substantially in the literature (see \cite{Ada2},  \cite{A1cnoi}, \cite{A2G}, \cite{A5Ponce}, \cite{CMR}, \cite{LeCoz}, \cite{DH}, \cite{FuJe}, \cite{fuoht}, \cite{ghw}, \cite{HMZ1}, \cite{KaOh} and  references therein). The case of Schr\"odinger models on star graphs with $\delta$  conditions on the vertex also have been studied recently in Adami {\it et. al} (\cite{AdaNoj15}, \cite{AdaNoj14} and references therein) and  Angulo\&Goloshchapova (\cite{AG1}, \cite{AG2}, \cite{A2G}). The case of either other defect type  or nonlinearity have been  studied in  \cite{Ada2}, \cite{A4Herna}, \cite{A2G}, \cite{A3G}. The  specific case $\lambda_1=2$,  $\lambda_2=-1$ and $Z>0$ was recently considered by Genoud\&Malomed\&Weishaupl in \cite{GMW}. 

The combination of nonlinearities in (\ref{deltasch1}) is well known in optical media (\cite{FilhoABL}, \cite{FilhoAR}, \cite{BouCLT}). In particular, we recall that for a effective linear
potential term, $V(x)$, the general NLS model
$$
iu_t+ u_{xx} + V(x)u +F(u)=0
$$
represents a trapping (wave-guiding) structure for light beams, induced by
an inhomogeneity of the local refractive index. In particular, the delta-function term in NLSCQ adequately represents a narrow trap which is able to capture broad solitonic beams.

The focus of this manuscript  is the existence and stability properties of  standing-wave solutions for the model NLSCQ, namely, solutions in the form
\begin{equation}\label{stand}
u(x,t)=e^{-i\omega t} \phi_\omega(x),\;\;\;\;\;\;\;\;\omega\in I\subset \mathbb R,
\end{equation}
where $I$ is  an interval and the profile $\phi_\omega:\mathbb R\to \mathbb R$ belongs to the domain of the quantum operator $-\partial_x^2-Z\delta(x)$, more exactly, for 
$$
\phi_\omega\in D(-\partial_x^2-Z\delta(x))=\{f\in H^1(\mathbb R)\cap H^2(\mathbb R-\{0\}): f'(0+)-f'(0-)=-Zf(0)\}.
$$
Thus, we have for $x\neq 0$, that the profile $\phi_\omega$ satisfies the elliptic equation
\begin{equation}\label{peak}
\phi''_\omega(x)-\omega \phi_\omega(x)-\lambda_1 \phi^3_\omega(x)-\lambda_2 \phi^5_\omega(x)=0.
\end{equation}
 The existence of solutions for equation (\ref{peak}) have been considered in analytic, numerical and experimental works for specific values of the parameters $Z, \omega, \lambda_1,  \lambda_2$. For $Z=0$, the rigorous existence and stability analysis of standing waves for NLSCQ model with general double-power nonlinearities can be found in \cite{Ma}-\cite{ota}. The existence and stability for $\lambda_1\neq 0$, $\lambda_2=0$, $\omega<0$,  $Z\neq 0$, and with the cubic nonlinearity term changed by the  single power-law nonlinearity $|\phi_\omega|^{p-1}\phi_\omega$, $p>1$, have been extensively discussed earlier in   \cite{A2G}, \cite{LeCoz}, \cite{FuJe}, \cite{fuoht}, \cite{ghw}, \cite{KaOh}. In a periodic framework, we have the recents works of Angulo \cite{A1cnoi} and Angulo\&Ponce \cite{A5Ponce}. 
 
 Now, it is well-known that for arbitrary values of $\lambda_1$ and  $\lambda_2$,  NSLCQ model can not have standing-wave solutions vanishing at the infinity (still in the case $Z=0$). Moreover,  in the case of the existence of solutions may happened that  exact solutions are not  available in general. Recently, Genoud\&Malomed\&Weishaupl in \cite{GMW} have studied the stability of  a family of explicit standing-wave solutions for NLSCQ model with  a combination of an attractive ($\lambda_1=2$) and repulsive ($\lambda_2=-1$) nonlinearities, and with a focusing $\delta$-interaction, $Z\in (0, \sqrt{3}) $, such that the wave phase velocity $-\omega$ satisfies $-\omega\in (\frac{Z^2}{4}, \frac34 +\frac{Z^2}{4})$. 
 Here, we extend and complete some  of the results in \cite{GMW} and we determine the profile  for (\ref{peak}) in the cases  attractive-attractive and attractive-repulsive.

The  peak-standing-wave solutions for NLSCQ model which we are interested here the following:
\begin{enumerate}
\item[1)] atractive-atractive ($\lambda_1, \lambda_2>0$): the parameters $\omega$ and  $Z$ in (\ref{peak}) will satisfy that $\omega<0$, $Z\in \mathbb R$, and the  condition $-\omega>\frac{Z^2}{4}$. Thus, we have the profile $\phi_{\omega}=\phi_{\omega, Z, \lambda_1, \lambda_2}$ for $\alpha=\frac{\lambda_1}{4}$ and $\beta=\frac{\lambda_2}{3}$
\begin{equation}\label{soldelta}
\phi_{\omega}(x)=\left[-\frac{\alpha}{\omega}+\frac{\sqrt{\alpha^2-\beta \omega}}{-\omega}\cosh
\left(2\sqrt{-\omega}\left(|x|+R^{-1}\left(\frac{Z}{2\sqrt{-\omega}}\right)\right)\right)\right]^{-\frac{1}{2}}
\end{equation}
with $R:(-\infty,\infty)\rightarrow(-1,1)$ being the diffeomorphism defined by 
\begin{equation}\label{R}
R(s)=\frac{\sqrt{\alpha^2-\beta \omega}\sinh(2\sqrt{-\omega}s)}{\alpha+\sqrt{\alpha^2-\beta \omega}\cosh(2\sqrt{-\omega}s)}.
\end{equation}
We note that for every $\lambda_2>0$, the expression $\sqrt{\alpha^2-\beta \omega}$ is well-defined.

\item[2)] attractive-repulsive ($\lambda_1>0$, $\lambda_2<0$): the profile $\phi_{\omega}$ in (\ref{soldelta}) is still a solution for (\ref{peak}), but we have the following restriction of parameters,
\begin{equation}\label{res}
\frac{Z^2}{4}<-\omega<-\frac{3\lambda_1^2}{16\lambda_2}, \qquad  |Z|< \frac{\sqrt{3}\lambda_1}{2\sqrt{-\lambda_2}}.
\end{equation}
\end{enumerate}

Figure 1 below shows the profile of $\phi_{\omega}$ for $\lambda_1=\lambda_2=1$ with $Z>0$ and $Z<0$.  Now, for the case of  repulsive nonlinearities ($\lambda_1, \lambda_2<0$) and $Z<0$, by using  a Pohazev type argument we can show  the nonexistence of nontrivial solutions for (\ref{peak}) vanishing at the infinity.

Theorem \ref{main} and Theorem \ref{main2}, below, establish our  stability results associated to the profile $\phi_{\omega}$ in (\ref{soldelta})  for the cases  $\lambda_1, \lambda_2 >0$ and $\lambda_1>0$, $\lambda_2 <0$. We note that in the case attractive-repulsive, we extend the results in \cite{GMW}. 

Now, the basic symmetry associated to equation \eqref{deltasch1} is the phase-invariance,  since the translation invariance of the solutions is not hold due to the defect. Thus,  our notion of stability (instability) will be based with regard to this symmetry and it is formulated as follows:
For $\eta>0$, let $\phi$ be a solution of \eqref{peak} and define the neighborhood
$$
U_\eta(\phi)=\Big\{v\in X: \inf_{\theta\in\mathbb R}\|v-e^{i\theta}\phi\|_X<\eta\Big\}.
$$

\begin{Definition}\label{dsta} The standing wave $e^{-i\omega t}\phi$ is (orbitally) stable by the
flow of the NLSCQ model \eqref{deltasch1} in $X$ if for any $\epsilon>0$ there exists $\eta>0$ such that for any $u_0\in U_\eta(\phi)$, the solution $u(t)$ of  \eqref{deltasch1} with $u(0)=u_0$ satisfies $u(t)\in U_\epsilon(\phi)$ for all $t\in \mathbb R$. Otherwise, $e^{-i\omega t}\phi$ is said to be (orbitally) unstable in $X$.
\end{Definition}

The space $X$ in Definition \ref{dsta} will be considered in our stability theory being as  $H^1(\mathbb{R})$ or $H_{even}^1(\mathbb{R})$ (The space of even fuctions on $H^1(\mathbb{R})$). Thus, our main stability results for  the peak-standing-wave profiles in (\ref{soldelta}) are the following.

\begin{Theorem}\label{main} 
For $\lambda_1=\lambda_2=1$ and $Z^*\approx-0.866025403784$, we have  for  $\omega<0$  such that $-\omega>\frac{Z^2}{4}$ and $\phi_{\omega}$ defined in (\ref{soldelta}) the following:
\begin{enumerate}
\item [a)] For $Z\geq0$, $e^{-i\omega t}\phi_{\omega,Z}$ is orbitally stable in $H^1(\mathbb{R})$.
\item [b)] For $Z\in(Z^*,0)$, $e^{-i\omega t}\phi_{\omega,Z}$ is orbitally unstable in $H^1(\mathbb{R})$.
\item [c)] For $Z\in(Z^*,\infty)$, $e^{-i\omega t}\phi_{\omega,Z}$ is orbitally stable in $H_{even}^1(\mathbb{R})$.
\item [d)] For $Z<Z^*$,  $e^{-i\omega t}\phi_{\omega,Z}$ is orbitally unstable in $H_{even}^1(\mathbb{R})$ and so in $H^1(\mathbb{R})$.
\end{enumerate}
\end{Theorem}

\begin{Remark} Our numerical simulations (see Figures 2, 3 and 4 below) showed us that to find exactly the threshold value $Z^*$ such that the mapping $\omega \to  -\|\phi_{\omega,Z^*}\|^2$ has a only critical point can be very tricky, but it is possible to see that in the neighborhood of $Z^*$ that map does not oscillate.  In Theorem \ref{main} we consider the specific attractive-values $\lambda_1=\lambda_2=1$ due  to the complexity of the formulas appearing in the analysis established in Section 5 below. But, our numerical simulations showed that a similar  stability behavior of the profile $\phi_{\omega}$ is obtained for  $\lambda_1, \lambda_2>0$  and a threshold value of $Z$, $Z^*=Z^*(\lambda_1, \lambda_2)$.
\end{Remark}

\begin{Theorem}\label{main2} 
Let $\lambda_1>0$ and $\lambda_2<0$. Then, for  $\omega<0$ and $Z\in \mathbb R$ such that 
$$
\frac{Z^2}{4}<-\omega<-\frac{3\lambda_1^2}{16\lambda_2}, \qquad  |Z|< \frac{\sqrt{3}\lambda_1}{2\sqrt{-\lambda_2}},
$$
we have that $\phi_{\omega}$ defined in (\ref{soldelta}) has the following property:
\begin{enumerate}
\item [a)] For $0<Z<\frac{\sqrt{3}\lambda_1}{2\sqrt{-\lambda_2}}$, $e^{-i\omega t}\phi_{\omega,Z}$ is orbitally stable  in $H^1(\mathbb{R})$.
\item [b)] For $-\frac{\sqrt{3}\lambda_1}{2\sqrt{-\lambda_2}}<Z<0$, $e^{-i\omega t}\phi_{\omega,Z}$ is orbitally unstable in $H^1(\mathbb{R})$.
\item [c)] For $-\frac{\sqrt{3}\lambda_1}{2\sqrt{-\lambda_2}}<Z<0$, $e^{-i\omega t}\phi_{\omega,Z}$ is orbitally stable  in $H_{even}^1(\mathbb{R})$.
\end{enumerate}
\end{Theorem}

\begin{Remark} 
The stability result in item a) -Theorem \ref{main2} with the specific restrictions $\lambda_1=2$, $\lambda_2=-1$, $\frac{Z^2}{4}<-\omega<\frac34$, and  $Z\in (0, \sqrt{3})$, it was obtained recently by Genoud\&Malomed\&Weishaupl in \cite{GMW}.
\end{Remark} 

Our approach for proving Theorem \ref{main} and Theorem \ref{main2}  will be based in the general framework developed by Grillakis\&Shatah\&Strauss \cite{grillakis1}-\cite{grillakis2}, for a Hamiltonian system which is invariant under a one-parameter unitary group of operators. We recall that this approach requires spectral analysis of certain self-adjoint Schr\"odinger operator and in particular by determining the Morse index  is one of the more delicate  issues in the approach. In particular, our strategy will be based in analytic perturbation theory. Now for $Z>0$, we present a new approach  for obtaining that index based in the extension theory of symmetric operators by von Neumann and Krein (see Appendix). 

We also note that for applying the  instability criterium in \cite{grillakis2} in our case (initially being of spectral instability), we need to justify nonlinear instability from a spectral instability behavior (see Remark \ref{nonstab}). We note that our argument (spectral instability $\to$ nonlinear instability) complements several instability results in the literature for the case $\lambda_2=0$ in (\ref{deltasch1}).

This paper is organized as follows. Section 2 is devoted to establish a local and global well-posedness theory for the NLSCQ model \eqref{deltasch1}. Section 3 describes the construction of the profile $\phi_{\omega}$ in (\ref{peak}). In Section 4 we establish the spectral theory information for applying the approach in \cite{grillakis1}, \cite{grillakis2}.
 In Section 6  we give  the proof of the stability/instability Theorems \ref{main}-\ref{main2}.

\section[Local and global well posedness]{Local and global well-posedness for the NLSCQ model}\label{boacolh}

In this section we discuss some results about the local and global well-posedness problem associated to the NLSCQ equation in $H^1(\mathbb{R})$,
\begin{equation}\label{cachy12}
\left\{
\begin{array}{lll}
\displaystyle iu_{t}-A_Zu+(\lambda_1 |u|^2+\lambda_2|u|^4)u=0,\\
u(0)=u_0\in H^1(\mathbb{R}),
\end{array}
\right.
\end{equation}
where, $A_Z$ is the self-adjoint operator defined fromally by
\begin{equation}\label{FormalExpression}
A_Z:=-\frac{d^2}{dx^2}-Z\delta(x).
\end{equation}
 We recall that expression in (\ref{FormalExpression}) can be understood as a family of self-adjoint operators to one-paremeter $Z\to A_Z$  with domain $D(A_Z)$,
\begin{equation}\label{dain23}
D(A_Z)=\left\{g\in H^1(\mathbb{R})\cap H^2(\mathbb{R}-\left\{0\right\})|g'(0+)-g'(0-)=-Zg(0) \right\},
\end{equation}
and such that $A_Z g(x)= -\frac{d^2}{dx^2}g(x)$ for $x\neq 0$. That family represents all the self-adjoint extensions associated to  the following closed, symmetric, densely defined linear operator (see \cite{Albeverio}): 
\begin{equation*}
\left\{
\begin{aligned}
A_0&=-\frac{d^2}{dx^2}\\
D(A_0)&=\{g\in H^2(\mathbb{R}): g(0)=0 \}.
\end{aligned}
\right.
\end{equation*}
Moreover, for $Z\in \mathbb R$ is well known  that the essential spectrum of $A_Z $ is the nonnegative real axis, $\Sigma_{ess}(A_Z)=[0,+\infty)$. For $Z>0$, $A_Z $ has exactly one negative, simple eigenvalue, i.e., its discrete spectrum $\Sigma_{dis}(A_Z )$ is $\Sigma_{dis}(A_Z )=\{{-Z^2/4}\}$, with a strictly (normalized) eigenfunction $
 \Psi_Z(x)=\sqrt{\frac{Z}{2}}e^{-\frac{Z}{2}|x|}$. For $Z\leqq 0$,  $A_Z $ has not discrete spectrum, $\Sigma_{dis}(A_Z )=\emptyset$. Therefore the 
operators $A_Z$ are bounded from below, more exactly,
\begin{equation}\label{boundbelo}
\left\{
\begin{aligned}
&A_Z\geq-Z^2/4, \hspace{0.5cm}&Z>0.\\ 
&A_Z\geq 0,\hspace{1.0cm}&Z\leq0.
\end{aligned}
\right.
\end{equation}

Our local well-posedness theory is the following.

\begin{Theorem}\label{cazi} For any $u_0\in H^1(\mathbb{R})$, there exists $T>0$ and a unique solution $u$ of  (\ref{cachy12}) such that $u\in C([-T,T]; H^1(\mathbb{R}))\cap C^1([-T,T]; H^{-1} (\mathbb{R}))$ and  $u(0)=u_0$.  Moreover, since the nonlinearity $F(u, \bar{u})=\lambda_1|u|^2u+\lambda_2 |u|^4u$ is smooth we obtain that the mapping
$$
u_0\in H^1(\mathbb{R}) \to u\in C([-T,T]; H^1(\mathbb{R}))
$$
 is smooth.

If an initial data $u_0$ is even the solution $u(t)$ is also even.
\end{Theorem}
\begin{proof}
The proof of  local existence of solutions can be obtained via the abstract result in Theorem 3.7.1 in \cite{Cazenave} and from the properties established in (\ref{boundbelo}). Here by convenience, we will use standard arguments of Banach's fixed point theorem. Thus, without loss of generality we consider $Z<0$, and we will give the principal steps of the method. Indeed, we consider the  mapping $P_{u_0}: C([-T, T], H^1(\mathbb R)))\longrightarrow C([-T,T], H^1(\mathbb R))$ given by 
\begin{equation}\label{duhamel}
P_{u_0}[u](t)=e^{-itA_Z}u_0+i\int\limits_0^te^{-i(t-s)A_Z} F(u(s), \overline{u(s)})ds
\end{equation}
where $e^{-itA_Z}$ represents the unitary group associated to  equation (\ref{cachy12}) which is given explicitly  by the formula (see \cite{AlbBrz95}, \cite{GavSch86}, \cite{DH}, \cite{HMZ1} for the case $Z>0$)
\begin{equation}\label{pospro1}
e^{-itA_Z}\psi(x)= e^{it\partial_x^2} (\psi\ast \tau_Z) (x) \chi^0_{+}(x) + \Big[e^{it\partial_x^2} \psi(x) + e^{it \partial_x^2} (\psi\ast \rho_Z) (-x) \Big ]\chi^0_{-}(x),
\end{equation}
where
$$
\rho_Z(x)=-\frac{Z}{2} e^{-\frac{Z}{2} x}\chi^0_{-}(x),\;\; \tau_Z(x)=\delta (x)+ \rho_Z(x).
$$
Here  $\chi^0_{+}$ and $\chi^0_{-}$  denote  the characteristic functions of $[0,+\infty)$ and $(-\infty, 0]$ respectively, and $e^{it\partial_x^2} 
$ denotes the free group of Schr\"odinger when  $Z=0$. Now, one of the delicate points in the analysis is to show that the map $P_{u_0}$ is well-defined. We start by estimating the nonlinear term $F(u(s), \overline{u(s)})$ Thus, by using  the one-dimensional Gagliardo-Nirenberg inequality, i.e. 
$$
\|u\|_{L^q}\leq C\|u'\|^{\frac12-\frac1q}\|u\|^{\frac12+\frac1q},\quad q>2
$$
where $C>0$, and the relation $|(|f|^{p-1} f)'|\leq C |f|^{p-1} |f'|$, one obtains using H\"older that for $u\in H^1(\mathbb R)$ 
\begin{equation}\label{pres_space}
||F(u(s), \overline{u(s)})||_{H^1(\mathbb R)}\leq C_1(||u||^3_{H^1(\mathbb R)}+ ||u||^5_{H^1(\mathbb R)}).
\end{equation}  

Next, using that for $x>0$ we obtain $\partial_x [e^{-itA_Z}\psi](x)= e^{-itA_Z}\psi'(x)$ and for $x<0$ the relation $\partial_x [e^{-itA_Z}\psi](x)=- e^{-itA_Z}\psi'(x) + 2 e^{it \partial_x^2} \psi'(x)$, the inequality (\ref{pres_space}), $L^2$-unitarity of $e^{-itA_Z}$ and  $e^{it \partial_x^2}$, we obtain 
\begin{equation}\label{neq_well_1}
||P_{u_0}[u](t)||_{H^1(\mathbb R)}\leq C_2||u_0||_{H^1(\mathbb R)}+C_3T\sup\limits_{s\in[0,T]}(||u(s)||^3_{H^1(\mathbb R)}+ ||u(s)||^5_{H^1(\mathbb R)}).
\end{equation}
where the positive constants $C_2, C_3$ do not depend on $u_0$. Thus, $P_{u_0}[u]\in H^1(\mathbb R)$. The  continuity property of $P_{u_0}[u](t)$ and the contraction property of $P_{u_0}$ is proved of standard way. Therefore, we obtain the existence of a unique solution for the Cauchy problem associated to   (\ref{deltasch1}) on $H^1(\mathbb R)$.

We note from (\ref{pospro1})  that for $u_0$  even  we obtain that $e^{-itA_Z}u_0$ is also even, so from Duhamel equation (\ref{duhamel}) and uniqueness we have that $u(t)$ is also even for every $t\in [-T, T]$.

Next, we recall that the argument based on the contraction mapping principle above has the advantage that it also shows that being the non linearity $F(u, \overline{u})$ smooth then it  regularity   is inherited by the mapping data-solution. In fact,  by following the ideas in Corollary 5.6 in Linares\&Ponce \cite{LP} we consider  for $(v_0, v)\in B(u_0;\epsilon)\times C([-T, T], H^1(\mathbb R)) $ the application
$$
\Gamma(v_0, v)(t)=v(t)- P_{v_0}[v](t),\qquad t\in [-T, T].
$$
Then, by the analysis above $\Gamma(u_0, u)(t)=0$, for all $t\in [-T, T]$. Now, since  $F(u, \overline{u})$ is smooth then $\Gamma$ is smooth. Therefore, by using the arguments for obtaining the local theory in  $H^1(\mathbb R))$ above  we can show that the operator $\partial_v\Gamma(u_0, u)$ is one-to-one and onto. Thus by the Implicit Function Theorem  there exists a smooth mapping $\Lambda: B(u_0;\delta)\to C([-T, T], H^1(\mathbb R))$ such that $\Gamma(v_0,  \Lambda (v_0))=0$ for all $v_0\in B(u_0;\delta)$. This argument establishes the smooth property of the mapping data-solution associated to equation   (\ref{deltasch1}).

\end{proof}

\begin{Remark} It is immediate that the proof of local well-posedness in Theorem \ref{cazi} can be extended to the nonlinearity $G(u, \bar{u})
=\lambda_1|u|^{p-1}u+\lambda_2 |u|^{q-1}u$, $p>1$ and $q>1$. Now, the smooth property of the mapping data-solution can be only to assure for $p-1$ and $q-1$ being an even integer.  For neither $p-1$ or $q-1$ not being an even integer we have that $G(u, \bar{u})$ is $C^n$ with $n=min\{[p], [q]\}$ and  so the mapping data-solution will be $C^n$.

\end{Remark}

With regarding to the existence of global solution for  (\ref{cachy12}), we have the existence of the following two conserved quantities: the energy
\begin{equation}\label{firstintegral3delta}
E(u)=\frac{1}{2}\int_{\mathbb{R}}|u_x|^2dx-\frac{\lambda_1}{4}\int_{\mathbb{R}}|u|^4dx-\frac{\lambda_2}{6}\int_{\mathbb{R}}|u|^6dx-\frac{Z}{2}|u(0)|^2.
\end{equation}
and the charge $Q$ 
\begin{equation}\label{secondintegral3delta}
Q(u)=\frac{1}{2}\int_{\mathbb{R}}|u|^2dx.
\end{equation}

Now, for the case $Z=0$ in  (\ref{deltasch1}) is well-known that the double-power nonlinearity induce restrictions on the existence global de solutions. The following Theorem shows that a similar picture is happening for $Z\neq 0$. 

\begin{Theorem}\label{gwpdel}
The Cauchy problem (\ref{cachy12}) is globally well-posedness  in $H^1(\mathbb{R})$ provide the norm of the initial data $u(0)=u_0$ is small in $L^2(\mathbb{R})$ in the case  $\lambda_1\neq 0$ and $\lambda_2>0$. 

For  $\lambda_1\neq 0$ and $\lambda_2<0$ we obtain  global existence of solutions without restriction on the size of the initial data.

\end{Theorem}
\begin{proof} Without loss of generality we consider the case $\lambda_1=\lambda_2=1$ in (\ref{cachy12}). Thus for  proving our result, it is enough to show that the $H^1(\mathbb{R})$-norm of the solution  $u(t)$  has  a {\it a priori} bounded. Indeed, from the conservation of the quantity $Q$, it is clear that the quantity $||u(t)\|=\|u_0\|$. Next, we show that $||u_x(t)||^2-Z|u(0,t)|^2$ is bounded. From  Gagliardo-Nirenberg inequality, Young inequality and from (\ref{firstintegral3delta}),
\begin{equation}\label{espriridel}
\begin{aligned}
||u_x(t)||^2-Z|u(0,t)|^2
&=2E(u(t))+\frac{1}{2}||u(t)||^4_{L^4}+\frac{1}{3}||u(t)||^6_{L^6}\\
&\leq 2E(u(t))+\frac{\epsilon}{2}||u_x(t)||^{2}+\frac{C_\epsilon}{2}||u(t)||^{6}+\frac{C}{3}||u_x(t)||^{2}||u(t)||^{4}\\
&= 2E(u(t))+\frac{C_\epsilon}{2}||u_0||^6+\left(\frac{\epsilon}{2}+\frac{C}{3}||u_0||^4\right)||u_x(t)||^2,
\end{aligned}
\end{equation}
where $\epsilon$ is positive (small) and $C_\epsilon>0$.  Now, by the conservation of the quantities $E$, we have
\begin{equation}\label{castdel}
[1-h(||u_0||)]||u_x(t)||^2-Z|u(0,t)|^2\leq 2E(u_0)+4C_\epsilon [Q(u_0)]^3,
\end{equation}
where $h(x)=\frac{1}{2}\epsilon+\frac{C}{3}x^4$. So, for $\epsilon$ small, we choose $u_0$ satisfying the condition 
$$
\frac{1}{2}\epsilon+\frac{C}{3}||u_0||^4<1,
$$
and therefore 
\begin{equation}\label{estfindel}
||u_x(t)||^2-\frac{Z|u(0,t)|^2}{1-h(||u_0||)}\leq\frac{2E(u_0)+ 4C_\epsilon [Q(u_0)]^3}{1-h(||u_0||)}.
\end{equation}
Thus, for $Z<0$ follows  immediately that $
 ||u_x(t)||^2\leqq M(\|u_0\|, \|u_0\|_1, Z)$. Now, from the inequality  $|u(0, t)|^2\leqq 2\|u_x(t)\|\|u(t)\|$, we have 
 for $Z>0$ and  from Young inequality that for every $\epsilon>0$ exists $C_1=C(\epsilon, \|u_0\|, Z)>0$ such that
 $$
\frac{Z|u(0,t)|^2}{1-h(||u_0||)}\leqq \epsilon  ||u_x(t)||^2 +C_1 \| ||u(t)||^2
 $$
 Therefore, from (\ref{estfindel}) we deduce the following boundedness 
 $$
 ||u_x(t)||^2\leqq N(\|u_0\|,\|u_0\|_1, Z).
 $$

The case  $\lambda_1\neq 0$ and $\lambda_2<0$ is immediate from the estimative $||u_x(t)||^2-Z|u(0,t)|^2
\leqq 2E(u(t))+\frac{|\lambda_1|}{2}||u(t)||^4_{L^4}$. It finishes the proof.
\end{proof}

\begin{Remark} 
\begin{enumerate}

\item[i)] We note that the restriction about the size of the initial data in Theorem \ref{gwpdel} is essentially due to the ``one-dimensional critical exponent'' in the nonlinearity $ |u|^4u$.

\item[ii)] Theorem \ref{gwpdel} does not give us information about the global existence of solutions for any size of the initial data $u_0$ ( $\lambda_1\neq 0$ and $\lambda_2>0$) and so the possibility of a blow up behavior of solutions may  exist for  specific initial data $u_0$. Indeed,  classical variational arguments ensure that for $Z=0$, $\lambda_2=1$ and $\lambda_1\in \{-1, 0, 1\}$ we will obtain for $u_0\in  H^1(\mathbb R)$ satisfying   
$$
\|u_0\|<\|Q\|,
$$ 
 the property of global in time of the solutions for (\ref{deltasch1}), where $Q$ is the ground state of the mass critical problem in (\ref{deltasch1}) with $\lambda_1=0$. Recently, Le Coz\&Martel\&Rapha\"el in \cite{LMR} have showed  that for $\lambda_1=-1$ and $\|u_0\|=\|Q\|$, then the solution of (\ref{deltasch1}) is global and bounded in $H^1(\mathbb R)$. Moreover, for $\lambda_1=1$ and $\|u_0\|=\|Q\|$,  we obtain the  existence of finite blow up solutions.

\item[iii)] In Theorem \ref{oxaca} below, we show the existence of global solutions in $H^1(\mathbb R)$ for the attractive-attractive case in (\ref{cachy12}) with a  initial data $u_0$ close to  the orbit generated  by the profile  $\phi_{\omega, Z}$ for $Z>0$, and in $H_{even}^1(\mathbb R)$ for $Z>Z^*$.

\end{enumerate}
\end{Remark}

\section{Existence of standing waves}\label{standing} 

In this section we deal with the deduction of the explicit solutions in  (\ref{soldelta})  for the NLSCQ model (\ref{deltasch1}) with $\lambda_1>0$ and $\lambda_2\neq 0$. We consider the  cases, $Z=0$ and $Z\neq 0$. For $Z=0$ in (\ref{peak}), we have that $\phi\equiv \phi_\omega$ satisfies the nonlinear elliptic equation 
\begin{equation}\label{ordendoisz0}
\phi''+\omega\phi+\lambda_1\phi^{3}+\lambda_2\phi^{5}=0.
\end{equation}
By using a quadrature procedure and considering the boundary condition for the profile $\phi\to 0$ as $|x|\to \infty$, we obtain that
\begin{equation}\label{ord1z0}
[\phi']^2+\omega\phi^2+2\alpha\phi^{4}+\beta\phi^{6}=0,
\end{equation}
here, $\alpha=\lambda_1/4$, $\beta=\lambda_2/3$. In order to obtain an explicit solution of equation (\ref{ordendoisz0}), we will assume $\omega<0$ and $0<\phi$. Then, via the substitution $y=\phi^{2}$ in (\ref{ord1z0}) we deduce that
\begin{equation}\label{inttrans}
\frac{1}{2\sqrt{-\omega}}\bigintssss{\frac{dy}{y\sqrt{1+2\alpha \omega^{-1}y+\beta \omega^{-1}y^2}}}=x.
\end{equation}
Next, with $c$ a positive constant we have the formula
\begin{equation*}\label{inttransit}
\bigintssss{\frac{dy}{y\sqrt{1+2\alpha \omega^{-1}y+\beta \omega^{-1}y^2}}}=-\ln\left[c\left(\frac{\alpha y+\omega}{-y}+\frac{\sqrt{\beta \omega y^2+2\omega\alpha y +\omega^2}}{y}\right)\right]
\end{equation*}
 So, for  $c=1/\sqrt{\alpha^2-\beta \omega}$, and recalling that
$$
 \text{arcosh}(x)=\ln(x-\sqrt{x^2-1}),\hspace{0.5cm}\text{for all}\hspace{0.5cm}x\geq 1,  
$$ 
we can rewrite the integral in (\ref{inttrans}) as 
\begin{equation}\label{inttransdo}
\bigintssss{\frac{dy}{y\sqrt{1+2\alpha \omega^{-1}y+\beta \omega^{-1}y^2}}}=-\text{arcosh}\left(\frac{\alpha y+\omega}{-y\sqrt{\alpha^2-\beta \omega}}\right),
\end{equation}
therefore, a solution of  equation (\ref{ordendoisz0}) is given implicitly  by 
$$
-\text{arcosh}\left(\frac{\alpha \phi^{2}+\omega}{-\phi^{2}\sqrt{\alpha^2-\beta \omega}}\right)=2\sqrt{-\omega}x,
$$
or by the formula,
\begin{equation}\label{denominator}
\phi(x)=\left[\frac{-\omega}{\alpha+\sqrt{\alpha^2-\beta \omega}\text{ cosh}(2\sqrt{-\omega}x)}\right]^{\frac{1}{2}}=\left[-\frac{\alpha}{\omega}+\frac{\sqrt{\alpha^2-\beta \omega}}{-\omega}\cosh
\left(2\sqrt{-\omega} x\right)\right]^{-\frac{1}{2}},
\end{equation}
with $-\omega>0$ and $ \alpha^2-\beta \omega>0$. Now, we  note that for $\lambda_2>0$, the solution $\phi$ in (\ref{denominator}) is well defined for all $\omega<0$. For  $\lambda_2<0$, the solution $\phi$ is well defined for $\omega$ satisfying
$$
0<-\omega<-\frac{3\lambda_1^2}{16\lambda_2}.
$$

Next, we proceed to calculate the solutions of equation (\ref{peak}) when $Z\neq 0$. 
The following lemma shows us some of the properties that a solution $\phi\in H^1(\mathbb{R})$ of  (\ref{peak}) must satisfy.

\begin{Lemma}\label{regul} Let $\phi\in H^1(\mathbb{R})$ be a solution of (\ref{peak}), then, $\phi$ satisfies the following properties,
\begin{subequations}
\begin{align}
&\phi\in C^j(\mathbb{R}-\{0\})\cap C(\mathbb{R}),\text{ }\text{ } j=1,2.\label{reg}\\
&\phi''(x)+\omega\phi(x)+\lambda_1\phi^{3}(x)+\lambda_2\phi^{5}(x)=0,\hspace{0.6cm}\text{for all}\hspace{0.6cm} x\neq 0. \label{eqdifx}\\
&\phi'(0+)-\phi'(0-)=-Z\phi(0).\label{condsal}\\
&\phi'(x),\phi(x)\rightarrow 0, \hspace{0.6cm}\text{if}\hspace{0.3cm} |x|\rightarrow\infty.\label{complimitado}
\end{align}
\end{subequations}
\end{Lemma}

\begin{proof}
The proof of this lemma follows the ideas of the proof of Lemma 3.1 in \cite{FuJe}. The properties (\ref{reg}) and (\ref{complimitado}) are proved by a standard boostrap argument, namely, for all $\xi\in C_0^{\infty}(\mathbb{R}\setminus\{0\})$, the function $\xi\phi$ satisfies 
$$
(\xi\phi)''+\omega(\xi\phi)=\xi''\phi+2\xi'\phi'-\lambda_1\xi\phi^3-\lambda_2\xi\phi^5
$$
in the sense of distributions. Since the right hand side of the previous identity is in $L^2(\mathbb{R})$, then $\xi\phi\in H^2(\mathbb{R})$, that is to say, $\phi\in H^2(\mathbb{R}\setminus \{0\})\cap C^1(\mathbb{R}\setminus\{0\})$. The equation (\ref{eqdifx}) follows from the fact that $C_0^{\infty}(\mathbb{R}\setminus\{0\})$ is dense in $L^2(\mathbb{R})$. In relation to (\ref{condsal}), it is enough to ``integrate'' (\ref{peak}) from $-\epsilon$ to $\epsilon$,
$$
\int_{-\epsilon}^{\epsilon}\phi'' dx+\omega \int_{-\epsilon}^{\epsilon}\phi dx +Z\int_{-\epsilon}^{\epsilon}\delta(x)\phi dx+\int_{-\epsilon}^{\epsilon}\lambda_1\phi^3+\lambda_2\phi^5 dx=0,
$$
and  by doing $\epsilon\to 0$, we obtain that $\phi'(0+)-\phi'(0-)=-Z\phi(0)$.
\end{proof}

Now, the function 
\begin{equation}\label{esest}
\phi_{s}(x):=\phi(|x|-s),\hspace{0.5cm}s\in\mathbb{R}, 
\end{equation}
with $\phi$ given in (\ref{denominator}), satisfies all the properties of the previous lemma except possibly the jump condition (\ref{condsal}). So, to ensure that $\phi_s$ satisfies that type of condition, we proceed as follows, since $\phi_s$ is an even function, condition (\ref{condsal}) can be rewritten as
\begin{equation}\label{jump}
\phi_s'(0+)=-\frac{Z}{2}\phi_s(0),\hspace{0.5cm}\text{or equivalently,}\hspace{0.5cm}\phi'(s)=\frac{Z}{2}\phi(s).
\end{equation} 
Hence, from (\ref{denominator}), we obtain that
\begin{equation}\label{translat}
\frac{\sqrt{\alpha^2-\beta \omega}\text{ sinh}(2\sqrt{-\omega}s)}{\alpha+\sqrt{\alpha^2-\beta \omega}\text{ cosh}(2\sqrt{-\omega}s)}=\frac{-Z}{2\sqrt{-\omega}},
\end{equation}  
then, if we define $R:(-\infty,\infty)\rightarrow (-1,1)$ by
\begin{equation}\label{Erre}
R(s)=\frac{\sqrt{\alpha^2-\beta \omega}\text{ sinh}(2\sqrt{-\omega}s)}{\alpha+\sqrt{\alpha^2-\beta \omega}\text{ cosh}(2\sqrt{-\omega}s)}, 
\end{equation}
then, we have that $R$ is an odd, increasing diffeomorphism between the intervals $(-\infty,\infty)$ and (-1,1). In particular, from the expression (\ref{translat}), we conclude that  $
\frac{Z^2}{4}<-\omega $ and 
\begin{equation}\label{ese}
s=R^{-1}\left(\frac{-Z}{2\sqrt{-\omega}}\right).
\end{equation}
Finally, from (\ref{denominator}), (\ref{esest}) and (\ref{ese}), we can conclude that for $\alpha=\lambda_1/4$, $\beta=\lambda_2/3$, the function $\phi_{\omega,Z}$ given by
\begin{equation}\label{numeratorwithZ1}
\phi_{\omega,Z}(x)=\left[\frac{\alpha}{-\omega}+\frac{\sqrt{\alpha^2-\beta \omega}}{-\omega}\text{ cosh}\left(2\sqrt{-\omega}\left(|x|+R^{-1}\left(\frac{Z}{2\sqrt{-\omega}}\right)\right)\right)\right]^{-\frac{1}{2}}
\end{equation}
is a solution of  equation (\ref{peak}) providing that:
\begin{enumerate}
\item[i)] for  $\lambda_1, \lambda_2 >0$: $\frac{Z^2}{4}<-\omega$.
\item[ii)] for  $\lambda_1>0$, $\lambda_2 <0$:  $\frac{Z^2}{4}<-\omega<-\frac{3\lambda_ 1^2}{16\lambda_2}$.
\end{enumerate}
We observe that if $Z=0$ in the previous formula, we recover the function $\phi$ given in (\ref{denominator}),  that is, $\phi_{\omega,0}=\phi$. 

Thus we can establish the following existence result of peak standing-wave solutions  for (\ref{deltasch1}).

\begin{Theorem}\label{solution} 
\begin{enumerate}
\item[i)] Let $\lambda_1, \lambda_2 >0$ in (\ref{peak}). Then for $Z\in \mathbb R$ and $\omega<0$ such that $-\omega>\frac{Z^2}{4}$ we have a smooth family of solutions for (\ref{peak}), $\omega\to \phi_{\omega, Z}$ given by the formula in (\ref{numeratorwithZ1}). Moreover, the mapping $Z\to \phi_{\omega, Z}$ is a real analytic function.

\item[ii)] Let $\lambda_1>0$ and $\lambda_2 <0$ in (\ref{peak}). Then for $Z\in \mathbb R$ and $\omega<0$ satisfying 
$$
\frac{Z^2}{4}<-\omega<-\frac{3\lambda_ 1^2}{16\lambda_2},
$$
we have a smooth family of solutions for (\ref{peak}), $\omega\to \phi_{\omega, Z}$ given by the formula in (\ref{numeratorwithZ1}). Moreover, the mapping $Z\in (-\frac{\sqrt{3}\lambda_1}{2\sqrt{-\lambda_2}}, \frac{\sqrt{3}\lambda_1}{2\sqrt{-\lambda_2}})\to \phi_{\omega, Z}$ is a real analytic function.
\end{enumerate}
\end{Theorem}

Figure \ref{avr-memmap}  below shows the profile of $\phi_{\omega, Z}$ in (\ref{numeratorwithZ1}) in  the case attractive-attractive ($\lambda_1=\lambda_2=1$). The picture of the profiles in  the case attractive-repulsive is the same.

\begin{figure}[htb]
  \centering
  
    \subfigure[$\phi_{\omega,Z}$: $\omega=-3$, $Z=-2$.]{
      \includegraphics[scale=0.8]{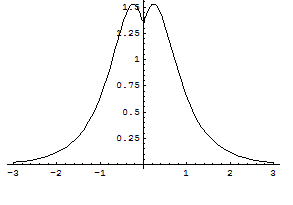}}
      \qquad
    \subfigure[$\phi_{\omega,Z}$: $\omega=-3$, $Z=2$.]{
      \includegraphics[scale=0.8]{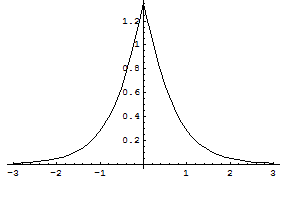}}
  \caption{The peak solutions $\phi_{\omega,Z}$ for $Z<0$ and $Z>0$.}
  \label{avr-memmap}
\end{figure}

\section{Spectral properties}\label{Spectro}

This section is dedicated to the study of the spectral properties of the operators associated to  the second variation  of the action functional $S_{\omega, Z}=E-\omega Q$ at the profile $\phi_{\omega}$ defined in (\ref{soldelta}). We note initially that $\phi_{\omega}$ is a critical point of $S_{\omega, Z}$. Next, we determine $S''_{\omega, Z}(\phi_{\omega})$. It consider
$u, v\in H^1(\mathbb R)$ such that $u=u_1+iu_2,\,\,v=v_1+iv_2$. Then,  we get 
\begin{equation*}\label{q_form}
\begin{split}
&S_{\omega, Z}''(\phi_{\omega})(u, v)=\int\limits_{\mathbb{R}}u'_1v'_1dx-\int\limits_{\mathbb{R}}u_1 v_1(\omega + 3\lambda_1\phi^2_{\omega} + 5 \lambda_2\phi^4_{\omega})dx -Z u_1(0)v_1(0) \\
&+\int\limits_{\mathbb{R}}u'_2v'_2dx-\int\limits_{\mathbb{R}}u_2 v_2(\omega + \lambda_1\phi^2_{\omega} + \lambda_2\phi^4_{\omega})dx -\gamma u_2(0)v_2(0).
\end{split}
\end{equation*}

Therefore, $S_{\omega, Z}''(\phi_{\omega})(u, v)$ can be formally rewritten as 
\begin{equation}\label{SBB}
S_{\omega, Z}''(\phi_{\omega})(u, v)=B_{1,Z}(u_1,v_1)+B_{2, Z}(u_2,v_2),
\end{equation}
where 
\begin{equation}\label{spec14}
\begin{split}
&B_{1, Z}(f,g)=\int\limits_{\mathbb{R}}f'g'dx-\int\limits_{\mathbb{R}}f g (\omega + 3\lambda_1\phi^2_{\omega} + 5\lambda_2 \phi^4_{\omega})dx-Z f(0)g(0),\\
&B_{2,Z}(f,g)=\int\limits_{\mathbb{R}}f'g'dx+\int\limits_{\mathbb{R}}f g (\omega + \lambda_1\phi^2_{\omega} +\lambda_2 \phi^4_{\omega})dx-Z f(0)g(0),
\end{split}
\end{equation} 
and $D(B_{j, Z})=H^1(\mathbb R)\times H^1(\mathbb R)$, $j\in\{1,2\}$.
Note that the forms $B_{j, Z}$, $j\in\{1,2\},$ are bilinear 
 bounded from below and closed. Therefore, by  the First Representation Theorem (see  \cite[Chapter VI, Section 2.1]{Kato}), they define operators $\mathcal L_{1, Z}$ and $\mathcal L_{2, Z}$ such that for $ j\in\{1,2\}$
 \begin{equation}\label{opera}
 \left\{ 
 \begin{split}
 &D(\mathcal L^\gamma_{j})=\{v\in H^1(\mathbb R): \exists w\in   L^2(\mathbb{R})\; s.t.\; \forall z\in H^1(\mathbb R), \;B_{j, Z}(v,z)=(w,z) \},\\
& \mathcal L_{j, Z}v=w.
 \end{split}\right.
\end{equation}
In the following theorem we  describe operators $\mathcal L_{1, Z}$ and $\mathcal L_{2, Z}$ in more explicit form.  The proof of this theorem follows the same lines as in Le Coz {\it et al.} \cite{LeCoz}.
\begin{Theorem}\label{repres}
The operators $\mathcal L_{1, Z}$ and $\mathcal L_{2, Z}$ determined in (\ref{opera}) are given by 
\begin{equation}\label{L12}
 \begin{split}
 \mathcal L_{1, Z}=-\frac{d^2}{dx^2} -\omega -3\lambda_1\phi^2_{\omega} - 5\lambda_2 \phi^4_{\omega}, \quad
 \mathcal L_{2, Z}=-\frac{d^2}{dx^2} -\omega -\lambda_1\phi^2_{\omega} - \lambda_2\phi^4_{\omega}
 \end{split}
\end{equation} 
on the domain $D_Z:= D\left({\cal L}_{i,Z}\right)=\left\{g\in H^1(\mathbb{R})\cap H^2(\mathbb{R}-\left\{0\right\})|g'(0+)-g'(0-)=-Zg(0) \right\}$.
\end{Theorem}

 Next, we proceed with the more  delicate part of our theory, it which is associated to finding the Morse index of the self-adjoint operators ${\cal L}_{1,Z}$ and ${\cal L}_{2, Z}$ defined in Theorem \ref{repres}. Here we will consider the parameters $\lambda_1, \lambda_2$, $\omega$ and $Z$ such that satisfy the relations in Theorem \ref{solution}. For finding this number we will use perturbation theory and we will follow the ideas in Le Coz {\it et.al} \cite{LeCoz}. We also give an alternative approach based in the extension theory for symmetric operators of von Neumman and Krein established recently by Angulo\&Goloshchapova (\cite{A2G}, \cite{A3G})  for finding this index at least in the case $Z>0$ (see Appendix below).

\begin{Theorem}\label{espil1} \textit{Spectral properties of ${\cal L}_{2,Z}$}. For $\omega<0$ and $Z$ satisfying  $\omega+Z^2/4<0$,  the self-adjoint linear operator ${\cal L}_{2,Z}:D_Z\rightarrow L^2(\mathbb{R})$ given in (\ref{L12}) has zero as a simple eigenvalue and $\phi_{\omega,Z}$ as its corresponding positive eigenfunction.  The rest of the spectrum is positive and away from zero. Additionally, $\sigma_{ess}({\cal L}_{2,Z})=[-\omega,\infty)$.
\end{Theorem}
\begin{proof} From (\ref{peak}) follows ${\cal L}_{2,Z}(\phi_{\omega,Z})=0$. Thus, since $\phi_{\omega,Z}>0$ we obtain from  the Sturm-Liouville oscillation theory extended to  operators with point interaction in  Angulo\&Goloshchapova \cite{A2G, A3G} that zero is a simple isolated eigenvalue, the remains of the spectrum is contained in $[\delta, +\infty)$ for $\delta>0$. Moreover,  from Weyl's theorem (see Reed\&Simon \cite{Reed}) we obtain the affirmation on the essential spectrum.
\end{proof}

Now, we study the kernel of ${\cal L}_{1,Z}$ for $Z\neq 0$.

\begin{Lemma}\label{Kert} For $Z\neq 0$ and  $\lambda_1, \lambda_2$ satisfying the conditions in Theorem \ref{solution}, the kernel of ${\cal L}_{1,Z}$ is trivial.
\end{Lemma}
\begin{proof} Let $v\in Ker({\cal L}_{1,Z})$, then we have 
\begin{equation}\label{valprop}
\left\{
\begin{aligned}
&{\cal L}_{1,Z}v(x)=0, \hspace{0.5cm} x>0.\\ 
&v\in L^2(0,\infty).
\end{aligned}
\right.
\end{equation}
Now, since the linear problem (\ref{valprop}) has dimension one (see \cite{BerShu91}) and  $\phi_{\omega}'$ satisfies (\ref{valprop}) then there exists $\alpha\in\mathbb{R}$ such that $v(x)=\alpha_0\phi'_{\omega}(x)$, for all $x>0$. A similar argument shows $v(x)=\beta\phi'_{\omega}(x)$, for $x<0$, with $\beta\in\mathbb{R}$. Next, from the continuity of $v$ and the parity of the function $\phi_{\omega}$, we deduce that $\alpha_0=-\beta$ and then we can rewrite $v$ as
\begin{equation}\label{parada}
v(x)=
\begin{cases}
\alpha_0\phi'_{\omega}(x), & \text{if}\hspace{0.3cm}  x\geq0,\\
-\alpha_0\phi'_{\omega}(x) & \text{if}\hspace{0.3cm} x<0.
\end{cases}
\end{equation}
Since $v\in {\cal D}({\cal L}_{1,Z})$ follows from (\ref{parada})   
\begin{equation}\label{nose}
v'(0+)-v'(0-)=\alpha_0\phi_{\omega}''(0+)+\alpha_0\phi_{\omega}''(0-)=-Z\alpha_0\phi_{\omega}'(0+).
\end{equation}
Now, we argue by contradiction. If $\alpha_0\neq0$, from (\ref{eqdifx}) and (\ref{nose}), we obtain that $\phi_{\omega}''(0+)=-Z/2\phi_{\omega}'(0+)$. Multiplying equation (\ref{eqdifx}) by $g'$ and integrating on the interval $(0,R)$, we get 
\begin{equation}
-\frac{1}{2}(g'(R))^2+\frac{1}{2}(g'(0+))^2-F(g(R))+F(g(0+))=0,
\end{equation}where $F(s)=\omega s^2/2+\lambda_1s^{4}/4+\lambda_2s^{6}/6$, by doing $R\rightarrow\infty$ and from (\ref{complimitado}), we obtain
\begin{equation}\label{tomala}
\frac{1}{2}(g'(0+))^2+F(g(0+))=0.
\end{equation}
Now, since $\phi_{\omega}$ satisfies equation (\ref{eqdifx}) then from (\ref{tomala}), we infer that $\frac{1}{2}(\phi_{\omega}'(0+))^2=-F(\phi_{\omega}(0))$. In addition, since $\phi_{\omega}$ is an even function, we obtain that $\phi_{\omega}'(0+)=\frac{-Z}{2}\phi_{\omega}(0)$. Therefore, we deduce that $\phi_{\omega}(0)>0$ is a zero of the following function 
\begin{equation}\label{poli}
P(s)=\frac{Z^2}{8}s^2+\omega\frac{s^2}{2}+\lambda_1\frac{s^{4}}{4}+\lambda_2\frac{s^{6}}{6}.
\end{equation} 
On the other hand, from equation (\ref{eqdifx}), we have that
\begin{equation}
\lim_{x\rightarrow 0+} \phi_{\omega}''(x)=\phi_{\omega}''(0+)=-\omega\phi_{\omega}(0)-\lambda_1\phi^{3}_{\omega}(0)-\lambda_2\phi^{5}_{\omega}(0),
\end{equation}
and since $\phi_{\omega}''(0+)=\frac{Z^2}{4}\phi_{\omega}(0)$, we deduce that $\phi_{\omega}(0)$ is a positive zero of the function 
\begin{equation}\label{poliester}
R(s)=\frac{Z^2}{4}s+\omega s+\lambda_1s^3+\lambda_2s^{5}.
\end{equation} 
Now, since $s_0=\phi_{\omega}(0)$ is a zero of both (\ref{poli}) and (\ref{poliester}), after some algebraic manipulations, we deduce that
\begin{equation}\label{identi}
s_0^{2}=-\frac{3\lambda_1}{4\lambda_2},
\end{equation}
and so we get  immediately a contradiction when $
\lambda_1, \lambda_2>0$. Now, in the case $
\lambda_1>0$ and  $\lambda_2<0$, we obtain from (\ref{soldelta})-(\ref{denominator}) the relation $\phi^2_{\omega}(0)\leqq \phi^2(0)$ and therefore
$$
-\frac{3\lambda_1}{4\lambda_2}\leqq -\frac{3\lambda_1}{4\lambda_2}\frac{1}{1+\frac{4}{\lambda_1}\sqrt{\frac{\lambda^2_1}{16}-\omega \frac{\lambda_2}{3}}},
$$
it which is a contradiction. Therefore, we conclude that $\alpha_0=0$ and then $v\equiv0$. It finished the proof.
\end{proof}

\begin{Remark} Our restrictions about the sign of the parameters $\lambda_1, \lambda_2$ are evidenced by the identity (\ref{identi}). Thus for the case $\lambda_1<0$ and $\lambda_2>0$ is not clear for us that the statement in Lemma \ref{Kert} is true.
\end{Remark}

Now, for starting our study based in perturbation theory, we establish the spectral structure of our ``limiting'' operator associated to ${\cal L}_{1,Z}$ when $Z\neq 0$.

\begin{Theorem}\label{espil2} \textit{Spectral properties of ${\cal L}_{1,0}$}. For $\omega<0$, we consider the self-adjoint linear operator ${\cal L}_{1,0}:H^2(\mathbb{R})\rightarrow L^2(\mathbb{R})$ given by
$$
{\cal L}_{1,0}= -\frac{d^2}{dx^2} -\omega -3\lambda_1\phi^2_{\omega, 0} - 5 \lambda_2\phi^4_{\omega, 0}
$$
with $\phi_{\omega, 0}$ being the profile $\phi_{\omega, Z}$ in the case $Z=0$. Then, ${\cal L}_{1,0}$ has a unique negative simple eigenvalue $-\lambda$, with $\lambda>0$. Zero is a simple eigenvalue with eigenfunction $\phi'_{\omega, 0}$. The rest of the spectrum is away from zero. Additionally, $\sigma_{ess}({\cal L}_{1,0})=[-\omega,\infty)$.
\end{Theorem}
\begin{proof} Since  ${\cal L}_{1,0}(\phi'_{\omega, 0})=0$ and  $\phi'_{\omega, 0}$ has a unique zero in $x=0$, we obtain immediately from the  classical Sturm-Liouville oscillation theory (see Berezin\&Shubin \cite{BerShu91}) the theorem.
 \end{proof}

Now, we show that the family of operators ${\cal L}_{1,Z}$ depends analytically of the variable $Z$, with $Z$ satisfying the conditions in Theorem \ref{solution}.

  \begin{Lemma}\label{lem6.5} As a function of the variable $Z$, $\left\{{\cal L}_{1,Z}\right\}$ is a real analytic family of self-adjoint operators of type (B) in the sense of Kato.
\end{Lemma}
\begin{proof}
From the theorem VII-4.2 in Kato \cite{Kato} and Reed and Simon \cite{Reed}, it is enough to show that the bilinear forms given in (\ref{spec14}) 
are real analytic of type (B), namely
\begin{enumerate}
\item The domain $D(B_{1, Z})$ of the forms ${B}_{1,Z}$ is independent of the parameter $Z$. In our case, $D({B}_{1,Z})=H^1(\mathbb{R})$ for $Z$ satisfying the conditions in Theorem \ref{solution}.
\item For each $Z$, ${B}_{1,Z}$ is closed and bounded from below. 
\item Since $\phi_{\omega, Z}$ and $R$ given in (\ref{soldelta}) and (\ref{Erre}), respectively, are analytic functions then $\phi_{\omega,Z}$ is an analytic function of $Z$. Thus, for each $v\in H^1(\mathbb{R})$ the function $Z\rightarrow {B}_{1,Z}(v,v)$ is analytic.   
\end{enumerate}
It finishes the proof.
\end{proof}

Next, we use the Kato-Rellich theorem to prove some specific properties of the second eigenvalue and eigenfunction of the operator ${\cal L}_{1,Z}$. Namely,
\begin{Lemma}\label{piomega}
There exist $Z_0>0$ and analytic functions $\Pi_2:(-Z_0,Z_0)\rightarrow \mathbb{R}$, $\Omega_2:(-Z_0,Z_0)\rightarrow L^2(\mathbb{R})$, such that 
\begin{enumerate}[(i)]
\item For each $Z\in(-Z_0,Z_0)$, $\Pi_2(Z)$ is the second eigenvalue of ${\cal L}_{1,Z}$, which is simple and $\Omega_2(Z)$ its corresponding eigenfunction.
\item $\Pi_2(0)=0$ and $\Omega_2(0)=\phi'$, with $\phi=\phi_{\omega, 0}$ given in (\ref{denominator}).
\item $Z_0$ can be chosen small enough such that the spectrum of ${\cal L}_{1,Z}$ with $Z\in(-Z_0,Z_0)$ is greater than $0$ except by the first 
$2$ eigenvalues.
\end{enumerate}
\end{Lemma}

\begin{proof} The proof is standard. Indeed, there is a positive $M$ such that $\sigma\left({\cal L}_{1,Z}\right)\cap(-\infty,-M)=\emptyset$ for $Z\in [-a,a]$, $a$ small enough. From Theorem \ref{espil2}, defining $\lambda_{1,0}=-\lambda$ and $\lambda_{2,0}=0$, we can separate  the spectrum $\sigma\left({\cal L}_{1,0}\right)$ of ${\cal L}_{1,0}$ into two parts $\sigma_0=\{-\lambda,0\}$  
and $\sigma_1=[-\omega,\infty)$ by a simple closed curve $\Gamma\subset\rho\left({\cal L}_{1,0}\right)$ such that $\sigma_0\subset \text{Int}(\Gamma)$ and $\sigma_1$ in its exterior. Here, $\text{Int}(\Gamma)$ denotes the interior of $\Gamma$. From Lemma \ref{lem6.5} we can see that ${\cal L}_{1,Z}$ converges to ${\cal L}_{1,0}$ as $Z\to 0$ in the generalized sense (see Kato \cite{Kato}). So, from Theorem IV-3.16 in Kato \cite{Kato}, we have that $\Gamma\subset\rho\left({\cal L}_{1,Z}\right)$ for $Z\in(-\epsilon_1,\epsilon_1)$, $\epsilon_1$ small enough, and $\sigma\left({\cal L}_{1,Z}\right)$ is also separated by $\Gamma$ into two parts such that the part of $\sigma\left({\cal L}_{1,Z}\right)$ inside $\Gamma$ consists of a finite set of eigenvalues with total algebraic multiplicity $2$.

Now, for $i=1,2$, and $\gamma>0$ small enough we define the circles $\Gamma_i=\{z\in\mathbb{C}:|z-\lambda_{i,0}|=\gamma\}$, such that $\Gamma_1\cap\Gamma_2=\emptyset$ and $\Gamma_i$ is in the interior of $\Gamma$. Thus from the nondegeneracy of the eigenvalues $-\lambda_{i,0}$, we obtain that there exists $0<Z_1<\epsilon_1$ such that for $Z\in(-Z_1,Z_1)$, $\sigma\left({\cal L}_{1,Z}\right)\cap \text{Int}(\Gamma_i)=\left\{\lambda_{i,Z}\right\}$, where $\lambda_{i,Z}$ are simple eigenvalues for ${\cal L}_{1,Z}$, furthermore, $\lambda_{i,Z}\to \lambda_{i,0}$, as $Z\to 0$. Applying the Kato-Rellich's theorem (Theorem XII.8 in \cite{Reed}) for each one of the simple eigenvalues $\lambda_{i,0}$, $i=1,2.$, we obtain the existence of a positive $Z_0<Z_1$, and two analytic functions $\Pi_2, \Omega_2$ defined on the intervals $(-Z_0,Z_0)$ satisfying the items (i),(ii) and (iii) of the theorem. It finishes the proof.
\end{proof}

Now, we proceed to count the number of negative eigenvalues of the operator ${\cal L}_{1,Z}$. First of all, we do this for $Z$ small.

\begin{Lemma}\label{lem6.8} 
There exists $0<r<Z_0$ such that $\Pi_2(Z)<0$, for any $Z\in(-r,0)$ and $\Pi_2(Z)>0$ for any $Z\in(0,r)$. Therefore, for $Z$ negative and small ${\cal L}_{1,Z}$ has exactly two negative eigenvalues and for $Z$ positive and small ${\cal L}_{1,Z}$ has exactly one negative eigenvalue.    
\end{Lemma}
\begin{Proof}
By application of Taylor theorem around $Z=0$, the functions $\Pi_2$ and $\Omega_2$ in Lemma \ref{piomega} can be written as
\begin{equation}\label{taylorpiomega}
\begin{aligned}
\Pi_2(Z)&=\beta Z+O\left(Z^2\right),\\
\Omega_2(Z)&=\phi'_{\omega,0}+Z\psi_0+O\left(Z^2\right),
\end{aligned}
\end{equation}
where $\phi'_{\omega,0}=\frac{d}{dx}\phi_{\omega,0}$, $\beta\in \mathbb{R}$, ($\beta=\Pi_2'(0)$) and $\psi_0\in L^2(\mathbb{R})$ $(\psi_0=\Omega_2'(0))$. To obtain our result, it is enough to show that $\beta>0$ or equivalently that $\Pi_2(Z)$ is an increasing function of the variable $Z$ around $Z=0$. Since the function $Z\rightarrow\phi_{\omega,Z}$ is analytic, then for $Z$ close to zero, we have that   
\begin{equation}\label{taylorsol}
\phi_{\omega,Z}=\phi_{\omega,0}+Z\chi_0+O\left(Z^2\right).
\end{equation}
where 
\begin{equation}\label{explicitchi}
\chi_0=\left.\frac{d}{dZ}\phi_{\omega,Z}\right|_{Z=0}.
\end{equation}
Now, from the equation (\ref{peak}), we have that for all $\psi\in H^1(\mathbb{R})$,
\begin{equation}\label{delemas}
\left\langle -\phi''_{\omega,Z}-\omega\phi_{\omega,Z}-\lambda_1\phi^{3}_{\omega,Z}--\lambda_2\phi^{5}_{\omega,Z},\psi \right\rangle=Z\phi_{\omega,Z}(0)\psi(0).
\end{equation}Taking derivative with respect to the variable $Z$ in (\ref{delemas}) and evaluating in $Z=0$, we get that
\begin{equation}\label{delemaduro}
\left\langle {\cal L}_{1,0}\chi_0,\psi\right\rangle=\phi_{\omega,0}(0)\psi(0).
\end{equation}In order to obtain $\beta$ as a function of the variable $Z$. We compute the quantity $\left\langle {\cal L}_{1,Z}\Omega_2(Z),\phi'_{\omega,0}\right\rangle$ in two different ways   
\begin{enumerate}[(1)]
\item Since ${\cal L}_{1,Z}\Omega_2(Z)=\Pi_2(Z)\Omega_2(Z)$, then from (\ref{taylorpiomega})
\begin{equation}\label{primform}
\left\langle {\cal L}_{1,Z}\Omega_2(Z),\phi'_{\omega,0}\right\rangle=\beta Z||\phi'_{\omega,0}||^2+O\left(Z^2\right).
\end{equation}
\item Since ${\cal L}_{1,Z}$ is selfadjoint, then $\left\langle {\cal L}_{1,Z}\Omega_2(Z),\phi'_{\omega,0}\right\rangle=\left\langle \Omega_2(Z),{\cal L}_{1,Z}\phi'_{\omega,0}\right\rangle$. Now, since ${\cal L}_{1,0}\phi'_{\omega,0}=0$, it follows from (\ref{explicitchi}) that  
\begin{equation}\label{ele}
\begin{aligned}
{\cal L}_{1,Z}\phi'_{\omega,0}&={\cal L}_{1,0}\phi'_{\omega,0}+\left[f'\left(\phi_{\omega,0}\right)-f'\left(\phi_{\omega,Z}\right)\right]\phi'_{\omega,0}\\
&=\left[f'\left(\phi_{\omega,0}\right)-f'\left(\phi_{\omega,0}+Z\chi_0+O(Z^2)\right)\right]\phi'_{\omega,0}\\
&=-f''\left(\phi_{\omega,0}\right)Z\chi_0\phi'_{\omega,0}+O\left(Z^2\right)
\end{aligned}
\end{equation}where $f(x)=\lambda_1x^3+\lambda_2x^{5}$. Thus, from (\ref{taylorpiomega}) and (\ref{ele}), we obtain that
\begin{equation}\label{masmas}
\begin{aligned}
\left\langle {\cal L}_{1,Z}\Omega_2(Z),\phi'_{\omega,0}\right\rangle &=-Z\left\langle\phi'_{\omega,0},f''\left(\phi_{\omega,0}\right)\chi_0\phi'_{\omega,0}\right\rangle+O(Z^2).
\end{aligned}
\end{equation}
On the other hand, by direct computation, we see that 
\begin{equation}\label{tal}
{\cal L}_{1,0}(-\omega\phi_{\omega,0}-f\left(\phi_{\omega,0}\right))=f''(\phi_{\omega,0})[\phi'_{\omega,0}]^2.
\end{equation}Using (\ref{delemaduro}), (\ref{masmas}), (\ref{tal}) and that $\phi_{\omega,0}$ satisfies equation (\ref{ordendoisz0}), we obtain that
\begin{equation}\label{fin}
\begin{aligned}
\left\langle {\cal L}_{1,Z}\Omega_2(Z),\phi'_{\omega,0}\right\rangle&=-Z\left\langle\chi_0,f''(\phi_{\omega,0})[\phi'_{\omega,0}]^2\right\rangle+O\left(Z^2\right)\\
&=-Z\left\langle-{\cal L}_{1,0}\chi_0,-\omega\phi_{\omega,0}-f\left(\phi_{\omega,0}\right)\right\rangle+O\left(Z^2\right)\\
&=-Z\phi_{\omega,0}(0)(-\omega\phi_{\omega,0}(0)-f(\phi_{\omega,0}(0)))+O\left(Z^2\right)\\
&=-Z\phi_{\omega,0}(0)(\phi''_{\omega,0}(0))+O\left(Z^2\right).
\end{aligned}
\end{equation} 
\end{enumerate}
Finally, from (\ref{primform}) and (\ref{fin}), we conclude that
\begin{equation}
\beta=-\frac{\phi_{\omega,0}(0)\phi''_{\omega,0}(0)}{||\phi'_{\omega,0}||^2}+O(Z).
\end{equation}
Hence $\Pi_2'(0)=\beta>0$ for $Z$ small. This completes the proof of the lemma. 
\cqd \\\\
\end{Proof}

Now, we are in position for counting the number of negative eigenvalues of $\mathcal L_{i,Z}$ for every $Z$ admissible. By using  a classical continuation argument based on the Riesz-projection and  denoting the number of negatives eigenvalues of $\mathcal L_{i,Z}$ by $n(\mathcal L_{i,Z})$, we have the following characterization.

\begin{Theorem}\label{Numberofnegeig} Let $\lambda_1, \lambda_2$, and $\omega, Z$ satisfy the conditions in Theorem \ref{solution}. Then, 
\begin{enumerate}
 \item For the case of being $Z$ negative, $n\left({\cal L}_{1,Z}\right)=2$.
 \item For the case of being $Z$ positive, $n\left({\cal L}_{1,Z}\right)=1$.
\end{enumerate}
\end{Theorem}
\begin{Proof}
The proof is based in Lemma \ref{Kert} above and the ideas in Le Coz {\it et al.}  \cite{LeCoz}
\end{Proof}

We finish this section with the following information about  the second eigenfunction $\Omega_2(Z)$ of the operator ${\cal L}_{1,Z}$ obtained in Lemma \ref{piomega}.

\begin{Proposition}\label{ExtAutPar}The eigenfunction $\Omega_2(Z)$ of the linear operator ${\cal L}_{1,Z}$ associated to the second eigenvalue $\Pi_2(Z)$ obtained in Lemma \ref{piomega} can be extended for all  $Z$ admissible. Moreover, $\Omega_2(Z)$ is an odd function.
\end{Proposition}
\begin{Proof}
The proof follows the same ideas as in \cite{LeCoz}.
\end{Proof}


\section{Convexity condition}\label{subconv}

In this section, we study the behavior  of the function $\omega \rightarrow -\partial_\omega||\phi_{\omega,Z}||^2$ with $\phi_{\omega,Z}$ given in (\ref{soldelta}). Due to the complexity of the formulas appearing in our calculations we need to use the  mathematical software {\it Mathematica}. We divide our analysis in the two cases:

\begin{enumerate}
\item[1)] $\lambda_1=\lambda_2=1$ (for the general case, $\lambda_1, \lambda_2>0$, similar analysis can be done): we have initially from (\ref{numeratorwithZ1}) that
\begin{equation}\label{integ}
\begin{aligned}
\int_{-\infty}^{\infty}\phi_{\omega,Z}^2(x)dx&=\int_{-\infty}^{\infty}\frac{1}{\alpha(\omega)+\beta(\omega)\text{cosh}(2\sqrt{-\omega}(|x|+b))}dx,\\
&=\frac{1}{\sqrt{-\omega}}\int_{2\sqrt{-\omega}b}^{\infty}\frac{1}{\alpha(\omega)+\beta(\omega)\text{cosh}(u)}du,
\end{aligned}
\end{equation} 
with $\alpha(\omega)=\frac{-1}{4\omega}$, $\beta(\omega)=\frac{\sqrt{9-48\omega}}{-12\omega}$, $b=R_\omega^{-1}\left(\frac{Z}{2\sqrt{-\omega}}\right)$ and $R_\omega$ given by
\begin{equation}\label{aux}
R_\omega(b)=\frac{\beta(\omega) \text{sinh}(2\sqrt{-\omega}b)}{\alpha(\omega)+\beta(\omega)\text{cosh}(2\sqrt{-\omega}b)}.
\end{equation}
By setting $\gamma(\omega)= \frac{\sqrt{3}}{\sqrt{3-16\omega}}$, we rewrite (\ref{integ}) in the following form 
\begin{equation}
\begin{aligned}
\int_{-\infty}^{\infty}\phi_{\omega,Z}^2(x)dx &= \frac{1}{\beta(\omega)\sqrt{-\omega}} \int_{2\sqrt{-\omega}b}^{\infty}\frac{1}{\gamma(\omega)+\text{cosh}(u)}du\\
&=-2\sqrt{3}\text{ }\text{arctan } \left(  \frac{\sqrt{3}-\sqrt{3-16\omega}}{4\sqrt{-\omega}} \text{ } \text{tanh} \left(  \frac{u}{2} \right )\right) \Biggr|_{u=2\sqrt{-\omega}b}^{u=\infty},
\end{aligned}
\end{equation}that is
\begin{equation}\label{normcuad}
||\phi_{\omega,Z}||^2=-2\sqrt{3}\left[\text{arctan}\left(\theta(\omega)\right)-\text{arctan}\left(\theta(\omega)\text{ }\text{tanh}(\sqrt{-\omega}b)\right)\right],
\end{equation}with $\theta(\omega)=\frac{\sqrt{3}-\sqrt{3-16\omega}}{4\sqrt{-\omega}}$. Now, we will obtain a formula to compute $\frac{db}{d\omega}$. From the relation (\ref{aux}) and setting  
\begin{equation}\label{agal}
H(\omega,b)=2\sqrt{-\omega}R_\omega(b),  
\end{equation} we deduce from (\ref{translat})-(\ref{Erre}) that $H(\omega,b)=Z$. Furthermore, since $\partial_bH>0$, then from the chain rule we deduce that
\[
\frac{db}{d\omega}=-\frac{\partial_\omega H}{\partial_bH},
\]therefore, using  (\ref{aux}) and (\ref{agal}), we conclude that 
\begin{equation}
\frac{db}{d\omega}(\omega,b)=\frac{4\sqrt{3}\sqrt{-\omega}h(\omega)b\text{ }\text{cosh}(s(\omega)b)+2\sqrt{3}(3-32\omega)\text{sinh}(s(\omega)b)+t(\omega)}{8(-\omega)^{\frac{3}{2}}\sqrt{h(\omega)}(h(\omega)+\sqrt{3}\sqrt{h(\omega)}\text{cosh}(s(\omega)b))},
\end{equation}where $t(\omega)=h(\omega)^{\frac{3}{2}}(2s(\omega)b+\text{sinh}(2s(\omega)b))$, $h(\omega)=3-16\omega$ and $s(\omega)=2\sqrt{-\omega}$.
Finally, from (\ref{normcuad}) we obtain that
\begin{multline}\label{derc}
\frac{1}{2\sqrt{3}}\partial_\omega ||\phi_{\omega,Z}||^2=\frac{3-\sqrt{3h(\omega)}}{u(\omega)\sqrt{-\omega h(\omega)}}+\\
\frac{\text{sech}^2(\sqrt{-\omega}b)\left[-2bu(\omega)\sqrt{-\omega}+\text{sinh}(2\sqrt{-\omega}b)(3-\sqrt{3h(\omega)})+4u(\omega)b'(\omega)(-\omega)^{\frac{3}{2}}\right]}
{2\sqrt{-\omega h(\omega)}(8\omega+\text{tanh}^2(\sqrt{-\omega}b)(-3+8\omega+\sqrt{3h(\omega)}))},
\end{multline}
where $u(\omega)=h(\omega)+\sqrt{3h(\omega)}$. 

Next, by using the computational software program Mathematica, we study the behavior of the functions $(\omega,Z)\rightarrow -||\phi_{\omega,Z}||^2$ and $(\omega,Z)\rightarrow -\partial_\omega||\phi_{\omega,Z}||^2$ given in (\ref{normcuad}) and (\ref{derc}), respectively. After some  two-dimensional plots by fixing one interval for $\omega$ and by choosing specific values for $Z$ we find the existence of a unique threshold value $Z^*$ of $Z$ such that the function $(\omega,Z)\rightarrow -\partial_\omega||\phi_{\omega,Z}||^2$  changes of signal. We also study  the derivate-function $(\omega,Z)\rightarrow -\partial_\omega||\phi_{\omega,Z}||^2$ via a three-dimensional surface and we also saw the possible location of $Z^*$.

Hence, after a delicated numerical study for determining the threshold value of  $Z$ associated to the mapping  $\omega \to -\partial_\omega||\phi_{\omega,Z}||^2$ we  can establish the following result.

\begin{Theorem}\label{convg} Let $\lambda_1=\lambda_2=1$ in (\ref{soldelta}) and it consider  $\omega, Z$ such that $\omega+\frac{Z^2}{4}<0$. Then there is a threshold  value $Z^*$,  $Z^*\approx-0.8660254$, such that the function $\omega \rightarrow -||\phi_{\omega,Z}||^2$ satisfies the following properties:
\begin{equation}\label{crestno}
\begin{cases}
-\partial_\omega||\phi_{\omega,Z}||^2>0, & \text{if $Z>Z^*$},\\
-\partial_\omega||\phi_{\omega,Z}||^2<0, & \text{if $Z<Z^*$},
\end{cases}
\end{equation}
\end{Theorem}

Next, we give some numerical justification which  led us to establish the above result. Thus, by instance, for $\omega\in(-6,-1)$, $Z=-0.86$, and $\omega\in(-6,-1)$, $Z=-0.9$, in Figure \ref{normcres}, (a)-(b), respectively, we see that there is a drastic change in the growth of  the function $\omega \rightarrow -||\phi_{\omega,Z}||^2$ for some value of the parameter $Z\in (-0.9,-0.8)$. In the same way, Figure \ref{normchange}  (three dimensional) showed us a remarkable change on the behavior of the function $\omega \rightarrow -\partial_\omega||\phi_{\omega,Z}||^2$. For considering $\omega\in(-50,-2)$, $Z\in(-0.9,-0.8)$ and $Z\in(-0.8,-0.7)$, the three dimensional plot was more conclusive about the existence of a threshold value of $Z^*$. Now, for fixing $Z=-0.86602$ and $Z=-0.86603$ (see Figure \ref{deristu}), we improved our numerical localization for $Z^*$.

\begin{figure}[htb]
  \centering
    \subfigure[$\omega\in(-6,-1)$, $Z=-0.86$.]{
      \includegraphics[scale=0.8]{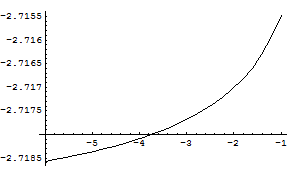}}\qquad
    \subfigure[$\omega\in(-6,-1)$, $Z=-0.9$.]{
      \includegraphics[scale=0.8]{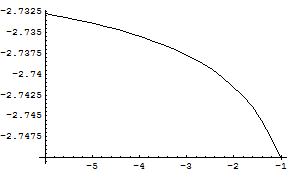}}
  \caption{Function $\omega\to-||\phi_{\omega,Z}||^2$ for specific values of $Z$.}
  \label{normcres}
\end{figure}

\begin{figure}[!h]
     \centering
     \subfigure[$\omega\in(-50,-2)$, $Z\in(-0.9,-0.8)$.]{
     \includegraphics[scale=0.75]{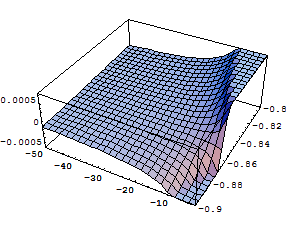}}\quad
     \subfigure[$\omega\in(-50,-2)$, $Z\in(-0.8,-0.7)$.]{
     \includegraphics[scale=0.75]{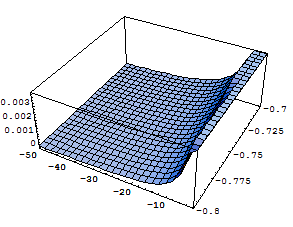}}
 \caption{Function $(\omega,Z)\to-\partial_\omega||\phi_{\omega,Z}||^2$, $\omega\in(-50,-2)$, $Z\in(-0.9,-0.8)$ and $Z\in(-0.8,-0.7)$ }
 \label{normchange}
\end{figure}
\begin{figure}[!h]
     \centering
 \subfigure[$\omega\in(-10,-0.8)$, $Z=-0.86602$.]{
     \includegraphics[scale=0.75]{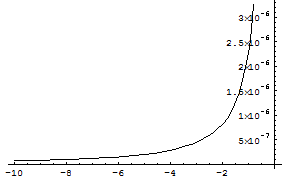}}\quad
     \subfigure[$\omega\in(-5,-1)$, $Z=-0.86603$.]{
     \includegraphics[scale=0.75]{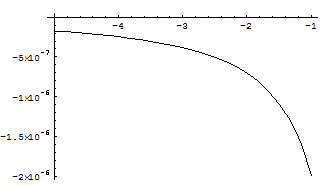}}
 \caption{Function $\omega\to-\partial_\omega||\phi_{\omega,Z}||^2$ for specific values of $Z$.  }		%
\label{deristu}
\end{figure}

\begin{Remark}
From our numerical study, $Z^*$ in Theorem \ref{convg} is the only critical point  of the mapping $\omega\to -||\phi_{\omega,Z}||^2$.
\end{Remark}

\item[2)]  $\lambda_1>0$ and $\lambda_2<0$:  since $Z$ and $\omega$ belong to a bounded interval, our numerical simulations are more accurate. Indeed, Figures 5 and 6 show the mapping $\omega \to - ||\phi_{\omega,Z}||^2$ for specific values of the parameters $\lambda_1>0$, $\lambda_2<0$, $Z$, and $\omega\in (-\frac34, -\frac14)$. These simulation showed us clearly that the mapping $\omega\to -\partial_\omega ||\phi_{\omega,Z}||^2$ is strictly positive. 
\begin{figure}[!h]
  \centering
	\subfigure[$\lambda_1=2$, $\lambda_2=-1$, $Z=-1$.]{
      \includegraphics[scale=0.8]{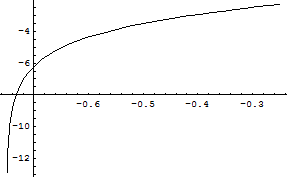}}\qquad
    \subfigure[$\lambda_1=2$, $\lambda_2=-1$, $Z=1$.]{
      \includegraphics[scale=0.8]{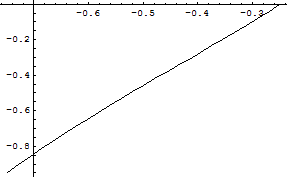}}
  \caption{The function $\omega\to-||\phi_{\omega,Z}||^2$, for $\omega\in (-\frac34,-\frac14)$ and $Z=\pm 1$.}
  \label{signonormaposneg}
   \end{figure}   
   
 \begin{figure}[!h]
  \centering
     \subfigure[$\lambda_1=4$, $\lambda_2=-2$, $Z=-1$.]{
      \includegraphics[scale=0.8]{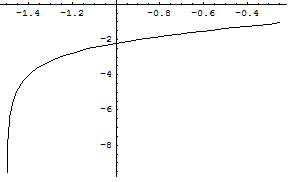}}\qquad
    \subfigure[$\lambda_1=4$, $\lambda_2=-2$, $Z=1$.]{
      \includegraphics[scale=0.8]{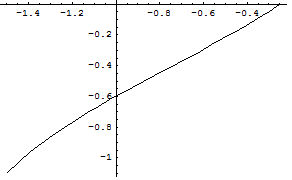}}
  \caption{The function $\omega\to-||\phi_{\omega,Z}||^2$, for $\omega\in (-\frac32,-\frac14)$ and $Z=\pm 1$.} 
  \label{signonormaposneg2}
\end{figure}

Therefore, from a detailed numerical study we can  conclude the following result. 

\begin{Theorem}\label{convg1} Let $\lambda_1>0, \lambda_2<0$ in (\ref{soldelta}) and it consider  $\omega, Z$ such that $\frac{Z^2}{4}<-\omega<-\frac{3\lambda_1^2}{16\lambda_2}$. Then, $-\partial_{\omega}||\phi_{\omega,Z}||^2>0$.
\end{Theorem}
\end{enumerate}
\section{Stability Theory}\label{secsta}

Our stability results, Theorem \ref{main} and Theorem \ref{main2}, it will based in the Instability Theorem and Stability Theorem in \cite{grillakis2}. Thus, by the sake of self-containdness we establish it.
\begin{Theorem}\label{st0} 
Let $Z \neq 0$ fixed and  
\begin{equation*}\label{st}
p_Z(\omega_0)=
\begin{cases}
\begin{aligned}
&1, \quad  {\text{if}} \;\; -\partial_\omega\|\phi_{\omega,Z}\|^2>0,\;\;at\;\; \omega=\omega_0,\\
&0, \quad  {\text{if}} \;\; -\partial_\omega\|\phi_{\omega, Z}\|^2<0,\;\;at\;\; \omega=\omega_0.
\end{aligned}
\end{cases}
\end{equation*}
We denote $\mathcal H_{\omega_0, Z}=S''_{\omega_0, Z}(\phi_{\omega_0, Z})$, so from the analysis in section \ref{Spectro} we have 
\begin{equation*}\label{eget}
\mathcal H_{\omega_0,Z}=
\begin{bmatrix}
{\cal L}_{1,Z} & 0 \\
0 & {\cal L}_{2,Z}
\end{bmatrix}.
\end{equation*}
Suppose that $Ker( L_{2,Z})=[\phi_{\omega_0, Z}]$, ${\cal L}_{2,Z}\geq 0$, and $Ker({\cal L}_{1,Z})=\{0\}$. Then from \cite{grillakis2} the following assertions hold:
\begin{itemize}
\item[$(i)$] If $n(\mathcal H_{\omega_0, Z})=p_Z(\omega_0)$, then the standing wave $e^{-i\omega_0 t}\phi_{\omega_0, Z}$  is orbitally  stable in  $H^1(\mathbb R)$. 
\item[$(ii)$] If $n(\mathcal H_{\omega_0, Z})-p_Z(\omega_0)$ is odd, then the standing wave $e^{-i\omega_0 t}\phi_{\omega_0, Z}$  is orbitally  unstable in   $H^1(\mathbb R)$. 
\end{itemize}

Analogous result holds for the case of changing the space $H^1(\mathbb R)$ by $H_{even}^1(\mathbb R)$. 
\end{Theorem}

\begin{Remark} \label{nonstab}The instability criterium part set up above deserves some comments:
\begin{enumerate}
\item[a)] it is well known from \cite{grillakis2}  that when $n(\mathcal H_{\omega_0, Z})-p_Z(\omega_0)$
 is odd, we obtain only spectral instability of $e^{i\omega t}\phi_{\omega_0,Z}$. Now, for obtaining orbital instability due to Theorem 6.1 in \cite{grillakis2}, it is sufficient to show estimate (6.2) in   \cite{grillakis2} for the semigroup  $e^{t\bb A_{\omega, Z}} $ generated by 
$$
\bb A_{\omega, Z}=\left(\begin{array}{cc}  0& \mathcal L_{2, Z} \\ -\mathcal L_{1,Z} &  0 \end{array}\right).
$$
For us, it is non-clear how to obtain that estimate (6.2).

\item[b)] In the particular case $n(\bb H_{\omega_0, Z})=2$ (it which will happened for the case $Z<0$)
we can to apply the results in Ohta \cite[Corollary 3 and 4]{ota2} for obtaining orbital instability part of the above theorem.  We note that in this case the instability results are obtained without using one argument through linear instability.

\item[c)] For justifying in a general framework (namely, $n(\bb H_{\omega_0, Z})=2$ no necessarily being true) orbital instability implications from a spectral instability result in the case of the model (\ref{deltasch1}), we can use the approach established in   \cite[Theorem 2]{HenPerWre}. The key point of this method is to use that  the mapping data-solution associated to model (\ref{deltasch1}) is at least of class $C^2$ (see Theorem \ref{cazi} above). We note that  the  results in  \cite{HenPerWre} have been applied successfully in the case of Schr\"odinger models on start graphs \cite{AG1}-\cite{AG2} and in \cite{ALN}-\cite{AN} for models of KdV-type.

\end{enumerate}

\end{Remark}

\begin{proof}$[${\bf{Theorem \ref{main}}}]

Let  $\omega$ such that $\omega+\frac{Z^2}{4}<0$. From Theorems \ref{espil1} and \ref{Kert} we have for every $Z\neq 0$ that $Ker( L_{2,Z})=[\phi_{\omega, Z}]$, ${\cal L}_{2,Z}\geq 0$, and $Ker({\cal L}_{1,Z})=\{0\}$.

\begin{enumerate}
\item[a)] For $Z>0$ we have from Theorem \ref{Numberofnegeig} and Theorem \ref{convg} that $n(\mathcal H_{\omega, Z})=p_Z(\omega)$ Thus from Theorem \ref{st0} we have that $e^{-i\omega t}\phi_{\omega, Z}$  is orbitally  stable in  $H^1(\mathbb R)$. We note that we
have initially a ``{\it {conditional stability}}''  for the profile $\phi_{\omega,Z}$, because of,  for $u_0\in U_\eta(\phi_{\omega,Z})$ we have that the solution $u(t)$ of (\ref{cachy12}) satisfies $u(t)\in U_\epsilon(\phi_{\omega,Z})$ for all $t\in (-T^*, T^*)$, where $T^*$ represents the maximal time of existence of the solution $u=u(t)$. But, as we will be shown below (Theorem \ref{oxaca}), $T^*=+\infty$.

\item[b)] Let $Z\in (Z^*,0)$. From Theorem \ref{convg} we have  $p_Z(\omega)=1$. Then from Theorem \ref{Numberofnegeig} follows that 
$n(\mathcal H_{\omega, Z})-p_Z(\omega)=2-1=1$. Thus we obtain from Remark \ref{nonstab}-$b)$ that $e^{-i\omega t}\phi_{\omega, Z}$  is orbitally  unstable in  $H^1(\mathbb R)$. By Remark \ref{nonstab}-$c)$ and Theorem \ref{cazi}, it nonlinear instability behavior can be deduced from a spectral instability result.

\item[c)] Let $Z\in (Z^*,+\infty)$. From Proposition \ref{ExtAutPar}, the second eigenvalue of  ${\cal L}_{1,Z}$ on  $L^2(\mathbb{R})$ has associated an odd eigenfunction. So, such eigenvalue disappears when  ${\cal L}_{1,Z}$ is restricted to the space $H^2_{even}(\mathbb{R})$ (we note that in this case ${\cal L}_{1,Z}: H^2_{even}(\mathbb{R})\to L^2_{even}(\mathbb{R})$). In addition, since $\phi_{\omega,Z}$ is an even function with  $\left\langle {\cal L}_{1,Z}\phi_{\omega,Z},\phi_{\omega,Z}\right\rangle<0$, for all $Z\neq 0$, then the first eigenvalue of the operator ${\cal L}_{1,Z}$ is still present. In other words, we have that $n(\mathcal H_{\omega,Z}|_{H^1_{even}(\mathbb{R})})=1$. Therefore, from the persistence of the solution for  (\ref{cachy12}) on the space $H^1_{even}(\mathbb{R})$ and $p_Z(\omega)=1$ we obtain, similarly as in item $a)$ above,  that $e^{-i\omega t}\phi_{\omega, Z}$  is  ``{\it {conditionally stable}}'' on $H^1_{even}(\mathbb{R})$. But, as we will be shown below (Theorem \ref{oxaca}), the solution really is global.

\item[d)] Let $Z<Z^*$. Following a similar analysis as in item $c)$ above, we can see that  $n(\mathcal H_{\omega,Z}|_{H^1_{even}(\mathbb{R})})=1$. Since $p_Z(\omega)=0$, follows from Remark \ref{nonstab}-$a)$ that the profile $e^{-i\omega t}\phi_{\omega, Z}$ is spectrally unstable. Therefore, from Theorem  \ref{cazi} and  Remark \ref{nonstab}-$c)$ follow that $e^{-i\omega t}\phi_{\omega, Z}$ is nonlinearly  unstable in  $H^1_{even}(\mathbb{R})$ and so in  $H^1(\mathbb{R})$.
\end{enumerate}
\end{proof}

\begin{proof}$[${\bf{Theorem \ref{main2}}}]

The proof is similar to that given for Theorem  \ref{main} above, by considering Theorem   \ref{cazi} (global existence for any initial data),   Theorem \ref{st0}, Remark \ref{nonstab} and Theorem \ref{convg1}. 
\end{proof}
\vskip0.1in

\begin{Remark}
As we do not have a global well posedness result of the Cauchy problem (\ref{cachy12}) for any initial data in the case $\lambda_1, \lambda_2>0$,  the orbital stability results given in the proof of Theorem \ref{main} above are already valid just for the existence time of solution $u(t)$. However, if we combine the local well posedness result in Theorem \ref{cazi} and our {\it {conditional stability result}} is possible to have the existence of global solutions of the Cauchy problem (\ref{cachy12}) for initial data  close to the orbit $\Omega_{\phi_{\omega,Z}}=\{e^{i\theta} \phi_{\omega,Z}: \theta\in [0,2\pi]\}$. 
\end{Remark}

\begin{Theorem}\label{oxaca}(\textbf{\textit{Global existence of solutions for  (\ref{deltasch1})   close to the orbit $\Omega_{\phi_{\omega,Z}}$}})

Let $\lambda_1, \lambda_2>0$. Then, all solutions $u=u(t)$ of equation (\ref{deltasch1}) there is for all time, provided:
\begin{enumerate}
\item[1)] the initial data $u(0)=u_0\in H^1(\mathbb R)$ is close to the orbit $\Omega_{\phi_{\omega,Z}}$ with   $Z>0$.

\item[2)]  the initial data $u(0)=u_0\in H_{even}^1(\mathbb R)$ close to the orbit $\Omega_{\phi_{\omega,Z}}$ with  $Z>Z^*$ .
\end{enumerate}

\end{Theorem}  
\begin{proof}
We will show that $u(t)$ is bounded in the $H^1(\mathbb R)$-norm. Indeed,  from the  {\it {conditional stability result}} of the orbit $\Omega_{\phi_{\omega,Z}}$ ($Z>0$ and $Z>Z^*$) established initially in the proof of Theorem \ref{main}, we have that for $\epsilon>0$, there exists $\eta>0$ such that for some $s\in \mathbb R$
\begin{equation}\label{esterm3}
||u(t)-e^{-is}\phi_{\omega,Z}||_1\leq\epsilon,\text{ }\text{ }\text{ }\text{ }\forall t\in (-T^*,T^*),
\end{equation} 
whenever the initial data $u(0)=u_0$ satisfies that $\inf_{s\in[0,2\pi]}||u_0-e^{-is}\phi_{\omega,Z}||_1<\eta$, with $T^*$ denoting the  maximal time of existence of  the solution $u=u(t)$ given by Theorem \ref{cazi}. Now, from the inequality (holds for all $s\in\mathbb{R}$ and $t\in (-T^*,T^*)$)
\begin{equation}\label{estbn1}
\begin{aligned}
||u(t)||_1\leq||u(t)-e^{-is}\phi_{\omega,Z}||_1+||e^{-is}\phi_{\omega,Z}||_1<\epsilon +||\phi_{\omega,Z}||_1,
\end{aligned}
\end{equation}
with  $\inf_{s\in[0,2\pi]}||u_0-e^{-is}\phi_{\omega,Z}||_1$ small enough,  we obtain the  boundedness of the solution $u$. This finishes the Theorem.
\end{proof}

\section{Appendix}

We note that for the case $Z>0$ we can use the theory of extension  for symmetric operators of  von Neumann and Krein (see \cite{Albeverio}, \cite{ak}, \cite{Nai67}, \cite{Reed}) for obtaining that the Morse-index for ${\cal L}_{1,Z}$ is exactly one in the cases $\lambda_1>0$ and $ \lambda_2>0$. For the cases $\lambda_1>0$ and $ \lambda_2<0$ that approach can not be  optimal with regard to 
the values of $\omega, Z$. Indeed, let $\mathcal A$ be a  densely defined symmetric operator in a Hilbert space $H$. The deficiency numbers of $\mathcal A$ are denoted by  $n_\pm(\mathcal A):=\dim\ker(\mathcal A^*\mp i\mathcal I)$, where $\mathcal A^*$ is the adjoint operator of $\mathcal A$ and  $\mathcal I$ is the  identity operator. To investigate the number of negative eigenvalues of $\mathcal L_{1, Z}$ we will use the following abstract result (see \cite[Chapter IV, \S 14]{Nai67}).

\begin{Proposition}\label{semibounded}
Let $\mathcal A$  be a densely defined lower semi-bounded symmetric operator (i.e., $\mathcal A\geq m\mathcal I$)  with finite deficiency indices $n_{\pm}(\mathcal A)=k<\infty$  in the Hilbert space $H$. Let also $\widetilde{\mathcal A}$ be a self-adjoint extension of $\mathcal A$.  Then the spectrum of $\widetilde{\mathcal A}$  in $(-\infty, m)$ is discrete and  consists of at most  $k$  eigenvalues counting multiplicities.
\end{Proposition}

Now, it is well known that  $A_Z=-\frac{d^2}{dx^2}- Z\delta(x)$ is the family of self-adjoint extensions   of the  symmetric  operator 
\begin{equation}\label{L0}
 \mathcal L^0=-\frac{d^2}{dx^2}, \quad D(\mathcal L^0)=\{ f\in  H^2(\mathbb R):  f(0)=0 \},
\end{equation} 
where $n_{\pm}(\mathcal L^0)=1$ (see \cite{Albeverio}). Now, by considering the minimal operator 
\begin{equation}\label{L0}
 \mathcal L_{min}=-\frac{d^2}{dx^2}-\omega-3\lambda_1\phi_{\omega, Z}^2-5\lambda_2 \phi_{\omega, Z}^4, \quad\quad D(\mathcal L_{min})=\{ f\in  H^2(\mathbb R):  f(0)=0 \},
\end{equation} 
we obtain from Theorem 6 in \cite{Nai67} that $ \mathcal L^0$ and  $\mathcal L_{min}$ have the same deficiency indices. Moreover, it is not difficult to see that $\mathcal L_{1, Z}$, for $Z\in \mathbb R$, it represents the family of  self-adjoint extensions   of the  symmetric  operator $ \mathcal L_{min}$.

Next, we see that $ \mathcal L_{min}\geq 0$. Indeed, since for $Z>0$ we have $\phi'_{\omega, Z} \neq 0$ for $x\neq 0$,  we can verify that for $f\in D(\mathcal \mathcal L_{min})$ we have 
 \begin{equation}\label{NLSL11}
\mathcal L_{min} f=\frac{-1}{\phi'_{\omega, Z}}\frac{d}{dx}\left[(\phi'_{\omega, Z})^2\frac{d}{dx}\left(\frac{f}{\phi'_{\omega, Z}}\right)\right],\quad\quad x\neq 0.
\end{equation}
 Now using \eqref{NLSL11} and integrating by parts, we get
\begin{equation}\label{NLSL12}
\begin{split}
(\mathcal L_{min} f,f)=&
\int\limits_{-\infty}^{0-}(\phi'_{\omega, Z})^2\left|\frac{d}{dx}\left(\frac{f}{\phi'_{\omega, Z}}\right)\right|^2dx\\ &+
\int\limits^{\infty}_{0+}(\phi'_{\omega, Z})^2\left|\frac{d}{dx}\left(\frac{f}{\phi'_{\omega, Z}}\right)\right|^2dx+
\left[f'\overline{f}-|f|^2\frac{\phi''_{\omega, Z}}{\phi'_{\omega, Z}}\right]_{0-}^{0+}.
\end{split}
\end{equation}
The integral terms in \eqref{NLSL12} are nonnegative.  Due to the condition $f(0)=0$, non-integral term vanishes, and we get $\mathcal L_{min}\geq 0$. Therefore from Proposition \ref{semibounded} we obtain $n(\mathcal L_{1, Z})\leq 1$.

Now,  from the relation  
 \begin{equation}\label{ineq1}
 (\mathcal L_{1, Z} \phi_{\omega,Z},  \phi_{\omega,Z})=(-2\lambda_1\phi^3_{\omega,Z}- 4\lambda_2\phi^5_{\omega,Z}, \phi_{\omega,Z}),
  \end{equation}
 we obtain that for $\lambda_1, \lambda_2>0$ follow $(\mathcal L_{1, Z} \phi_{\omega,Z},  \phi_{\omega,Z})<0$ and from the min-max principle we obtain $n(\mathcal L_{1, Z})\geqq1$ and therefore $n(\mathcal L_{1, Z})=1$ for all $\omega<0$ and $-\omega>\frac{Z^2}{4}$. 
 
 Next for the case $\lambda_1>0$ and $\lambda_2<0$, follows from Theorem \ref{espil2} and Theorem \ref{solution} that
  $$
 (\mathcal L_{1, Z} \chi_0, \chi_0)\to  (\mathcal L_{1,0 } \chi_0, \chi_0)<0,\quad\text{for}\;\;Z\to 0^+
 $$
 where $\chi_0\in H^2(\mathbb R)$ is a negative direction associated to $ \mathcal L_{1,0 }$. Thus, for $Z$ small enough we have  $n(\mathcal L_{1, Z})=1$. For an arbitrary $Z\in (0, \frac{\sqrt{3}\lambda_1}{2\sqrt{-\lambda_2}})$
  was difficult to find a negative direction $v$ such that $(\mathcal L_{1, Z} v, v)<0$. Moreover, numerical calculations showed us that the quantity in (\ref{ineq1}) is not always negative. Thus, we will establish that at least for  $0<Z<\frac{\sqrt{3}\lambda_1}{2\sqrt{-\lambda_2}}$ and $\omega$ such that
\begin{equation}\label{rest}
\frac{Z^2}{4}<-\omega<\min\left\{-\frac{3\lambda_1^2}{16\lambda_2},-\frac{1}{6}\frac{\lambda^2_1}{\lambda_2}+\frac{Z^2}{4}\right\},
\end{equation} 
we also have that $(\mathcal L_{1, Z} \phi_{\omega,Z}, \phi_{\omega,Z})<0$. Indeed, first of all, we can rewrite (\ref{ineq1}) in the following form, 
\begin{equation}\label{cineq1}
 (\mathcal L_{1, Z} \phi_{\omega,Z},  \phi_{\omega,Z})=-4\lambda_2\left(\frac{1}{2}\frac{\lambda_1}{\lambda_2}+\phi^2_{\omega,Z}, \phi^4_{\omega,Z}\right).
 \end{equation}
Secondly, for $Z>0$, we have that
\begin{equation}\label{cineq2}
\phi^2_{\omega,Z}(x)\leq\phi^2_{\omega,Z}(0),\hspace{0.5cm}\text{for all}\hspace{0.5cm} x\in\mathbb{R}.
\end{equation}
Now, $\phi_{\omega,Z}^2(0)$ in (\ref{cineq2}) can be computed analytically by solving the equation $P(\phi_{\omega,Z}(0))=0$ in (\ref{poli}). After some calculations, we obtain that
\begin{equation}\label{cinest}
\phi^2_{\omega,Z}(0)=\frac{-3\lambda_1}{4\lambda_2}\left(1-\sqrt{1-\frac{16\lambda_2}{3\lambda_1^2}\left(\omega+\frac{Z^2}{4}\right)}\right).
\end{equation}
From (\ref{rest}) and (\ref{cinest}), after some algebraic manipulations, we can infer that
\begin{equation}\label{cineq3}
\frac{1}{2}\frac{\lambda_1}{\lambda_2}+\phi^2_{\omega, Z}(0)<0.
\end{equation}
Lastly, from (\ref{cineq1}), (\ref{cineq2}) and (\ref{cineq3}), we deduce that $(\mathcal L_{1, Z} \phi_{\omega,Z},  \phi_{\omega,Z})<0$. Hence, $n(\mathcal L_{1, Z})=1$ for $\omega, Z$ satisfying (\ref{rest}).

\indent\textbf{Acknowledgements:} J. Angulo was supported partially
by  Grant CNPq/Brazil. 

\renewcommand{\refname}{References}

\end{document}